\documentclass[final]{cvpr}

\usepackage{times}
\usepackage{epsfig}
\usepackage{graphicx}
\usepackage{amsmath}
\usepackage{amssymb}

\usepackage{subfig}

\newtheorem{theorem}{Theorem}

\newtheorem{proposition}[theorem]{Proposition}
\newtheorem{assumption}{Assumption}
\newtheorem{lemma}{Lemma}
\newcommand{\be}{\begin{equation}}
\newcommand{\ee}{\end{equation}}
\newcommand{\bee}{\begin{equation*}}
\newcommand{\eee}{\end{equation*}}

\newcommand{\beaa}{\begin{eqnarray*}}
\newcommand{\eeaa}{\end{eqnarray*}}

\newcommand{\Prob}{\mathbb{P}}
\newcommand{\Rn}{\mathbb{R}}

\newcommand{\samples}{\mathcal{S}}
\newcommand{\ovec}{\operatorname{vec}}

\usepackage{enumerate}
\usepackage[algo2e,linesnumbered,vlined,ruled]{algorithm2e}

\newcommand{\iprod}[2]{\langle #1, #2 \rangle}

\newcommand{\te}[1]{\theta^{#1}}
\newcommand{\Expe}{\mathbb{E}}

\newcommand{\cF}{\mathcal{F}}
\newcommand{\bE}{\mathbb{E}}

\newcommand{\mbR}{\mathbb{R}}

\newcommand{\pp}[1]{\left(#1\right)}
\newcommand{\mfF}{\mathbf{F}}
\newcommand{\mbE}{\mathbb{E}}

\usepackage{booktabs}
\newenvironment{proof}{{\noindent\it Proof.}}{\hfill $\square$\par}
\newenvironment{theorem2}{{\\\noindent\it \rm {\textbf{Theorem.}}}}{\hfill \par}

\usepackage[pagebackref=true,breaklinks=true,colorlinks,bookmarks=false]{hyperref}

\def\cvprPaperID{7829} 
\def\confYear{CVPR 2021}

\begin{document}

\title{ Enhance Curvature Information by Structured Stochastic Quasi-Newton Methods}
%

\author{%
    Minghan Yang\textsuperscript{1,2}, ~~%
    Dong Xu\textsuperscript{1,2}, ~~%
    Hongyu Chen\textsuperscript{1}, ~~
    Zaiwen Wen\textsuperscript{2,3,4}, ~~
    Mengyun Chen\textsuperscript{5}
    \\
     {\textsuperscript{1} School of Mathematical Sciences, Peking University, China
} ~\\
    {\textsuperscript{2} Beijing International Center for Mathematical Research, Peking University, China} ~\\
 {\textsuperscript{3} Center for Data Science, Peking University, China} ~\\
 {\textsuperscript{4} National Engineering Laboratory for Big Data Analysis and Applications, Peking University, China} ~\\
 {\textsuperscript{5} Huawei Technologies Co. Ltd, China
} ~
  \\
 {\tt\small \{yangminghan, taroxd, hongyuchen, wenzw\}@pku.edu.cn ,chenmengyun1@huawei.com} \\
}

\maketitle

\begin{abstract}
In this paper, we consider stochastic second-order methods for minimizing a  finite summation of nonconvex functions. One important key  is to find an ingenious but cheap scheme to incorporate local curvature information.
Since the true Hessian matrix is often a combination of a cheap part and an expensive part, we propose a structured stochastic quasi-Newton method by using partial Hessian information as much as possible. By further exploiting either the low-rank structure or the kronecker-product properties of the quasi-Newton approximations, the computation of the quasi-Newton direction is affordable. Global convergence to stationary point and local superlinear convergence rate are established under some mild assumptions. Numerical results on logistic regression, deep autoencoder networks and deep convolutional neural networks show that our proposed method is quite competitive to the state-of-the-art methods.
\end{abstract}
\section{Introduction}
Consider the large-scale finite-sum optimization problem:
\be
\label{finite-sum}
 \vspace{-0.5ex}\min_{\theta \in \Rn^n}\Psi(\theta) = \frac{1}{N} \sum_{i=1}^N \psi_i(\theta) , \vspace{-0.5ex}\ee

where $\theta$ is the decision variable, $\psi_i$ is the component function  and $N$ is the number of functions. For many cases, $\psi_i$ is related to the data point $(x_i,y_i)$, i.e., 
\be
\label{finite-sum1}
 \vspace{-0.5ex}\psi_i(\theta) = \ell \left ( f(x_i, \theta),y_i \right ), \vspace{-0.5ex}
\ee
where $f(x_i, \theta)\in \Rn^m$ is the output of $x_i$ through certain model and $\ell$ is the loss function to measure the prediction error. 
The numbers $n$ and $N$ can be very huge. For example, the number of parameters $n$ is 175 billion in GPT-3 \cite{brown2020language}. Problem  (\ref{finite-sum}) widely arises
in many applications such as deep learning
\cite{Goodfellow-et-al-2016, lecun2015deep,Sch15} and
statistical learning \cite{friedman2001elements,vapnik2013nature}.

The first-order type methods are dominant approaches for the problem \eqref{finite-sum}.
The classical stochastic gradient descent method
(SGD) \cite{RobMon51} falls into this type and 
extensions of SGD have been widely studied in \cite{All17-Kat,JohZha13,nguyen2017sarah}
for better practical performance  and
theoretical properties. 

Recently, stochastic second-order methods have gained increasing attention. Since the computation and inversion of the Hessian matrix in the large-scale
applications is costly, various strategies to approximate the Hessian have been
developed in \cite{byrd2011use,milzarek2019stochastic,pilanci2017newton,RooMah19,XuRooMah19,yangSEQN}. 
In deep learning, stochastic second-order methods are expected to address the
scalability issue with large batch size.  The Hessian-free method  \cite{HF-DL}
uses the conjugate-gradient (CG) method to obtain a descent direction by
utilizing Hessian-vector products.  
The Gauss-Newton matrix
is investigated to approximate the Hessian matrix in \cite{gauss-Newton} and a practical
block-diagonal approximation to the Gauss-Newton matrix is studied in
\cite{practicalGN}. 
The so-called KFAC method developed in \cite{grosse2016kronecker}
utilizes the kronecker-factored approximation to Fisher information matrix (FIM). Its efficiency has been demonstrated in large-scale distributed parallel computing \cite{35epochsKFAC}. A Newton method for convolutional neural networks (CNN) is investigated in \cite{wang2020newton}.

 In this paper, we consider a structured quasi-Newton (QN) framework to enhance curvature information. This idea has been studied in a few second-order methods.
For example, on the nonlinear least squares problems with large residual, approaches that
compensate the Gauss-Newton matrix by a quasi-Newton approximation to the
complicate part of the Hessian matrix are often much better
\cite{nocedal2006numerical, sun2006optimization, zhou2010global}. This concept has been further verified
in optimization problems with orthogonality constraints in
\cite{hu2019structured}. In this paper, we develop the structured quasi-Newton method in the stochastic
 setting for problem \eqref{finite-sum}.


\subsection{Contribution}
Our main contributions are as follows.

    (1) The structures where the Hessian matrix is a summation of a
     cheap part and an expensive 
    part are exploited for machine learning
    problems. The concept of structured quasi-Newton is extended to stochastic setting. 

(2) A general structured stochastic  quasi-Newton framework is proposed for
    the large-scale finite-sum problem \eqref{finite-sum}. We formulate stochastic secant conditions based on the partial Hessian matrix, then various quasi-Newton matrices can be constructed. By further exploiting either the low-rank structure or the kronecker-product properties of the quasi-Newton approximations, the computation of the refined direction is affordable. 

    (3)   Global
    convergence is established if the step sizes are chosen properly and the
    stochastic errors satisfy certain summability conditions. A local
    superlinear convergence rate is also guaranteed for the structured stochastic quasi-Newton method if the sample size is sufficiently large.
\section{Structures of the Hessian Matrices} \label{sec:stru-Hessian}
In this part, we assume that the Hessian matrix $\nabla^2 \Psi(\theta)$ can be divided into two
different parts:
\be \label{HStru} \nabla^2 \Psi(\theta) = H(\theta) + \Pi(\theta), \ee
where the part $H(\theta)$ is relatively cheap and accessible while the other part
$\Pi(\theta)$ is expensive or even not computable.

 Specifically, there are a few possible situations. \textbf{(1) The dimension $n$ is so high} that an explicit storage of $\nabla^2 \Psi(\theta)$ is
prohibitive. Hence, it is favorable to utilize an low-rank approximation to the Hessian matrix \cite{roux2008topmoumoute} or just use the Hessian information implicitly \cite{HF-DL}. \textbf{(2) The number of data points $N$ is tremendous.} Even if each component $\nabla^2 \psi_i(\theta)$ is cheap,  assembling all parts becomes a non-negligible task when the size $N$ is huge. \textbf{(3) The derivatives are complicated.} In certain cases, the explicit expression of the Hessian can not be derived or computed easily. 

In many applications, the reasons listed above are mixed.
We next explain a few typical scenarios of \eqref{HStru}. The goal is to explore $\Pi(\theta)$ for better performance at a relatively low computational cost.

\subsection{Hessian Matrices for General Problem \eqref{finite-sum}}
The subsampling procedure only selects a small fraction of the data in certain ways for
the update.  Given a subset $ \mathcal{S}_{H} \subseteq
\{1,2,\dots,N\}$, the subsampled Hessian matrix is defined as
 $\nabla^2_{\mathcal{S}_H}\Psi(\theta) :=\frac{1}{|\mathcal{S}_H|}
\sum_{i\in\mathcal{S}_H} \nabla^2 \psi_i(\theta) .$ It is common to choose the subset $ \mathcal{S}_{H}$  by uniform random sampling. 
In this case, if $|\mathcal{S}_{H}|$ is small, let
\be
\label{sub-Hessian}
H(\theta) = \nabla^2_{\mathcal{S}_H}\Psi(\theta),
\ee and we have $\Pi(\theta)= \nabla^2
\Psi(\theta)-\nabla^2_{\mathcal{S}_H}\Psi(\theta)$ correspondingly.


\subsection{Hessian Matrices for Format \eqref{finite-sum1}}
By using the chain rule twice, the Hessian matrix of the case \eqref{finite-sum1} can
be split as:
\begin{align}
    H(\theta)&=\frac{1}{N} \sum_{i=1}^N H_i(\theta) = \frac{1}{N} \sum_{i=1}^N J_f^i(\theta)  \nabla_f^2 \ell_i(\theta)
(J_f^i(\theta))^\top, \label{GGN-mtx}\\
\Pi(\theta)&=\frac{1}{N} \sum_{i=1}^N \Pi_i(\theta)  = \frac{1}{N} \sum_{i=1}^N\sum_{j=1}^m \nabla_{f_j}
\ell_i(\theta)\nabla_\theta^2f_j^i(\theta), \label{GGN2-mtx}
\end{align}
where $J_f^i(\theta)=\nabla_\theta f(x_i,\theta)\in \mbR^{n\times m}$ and
$f_j^i(\theta)$ is the $j$-th component of $f_i(\theta) := f(x_i, \theta)$.
The term $H(\theta)$ here is also called the generalized Gauss-Newton (GGN) matrix,
which is a good approximation to the Hessian matrix. It is positive semi-definite (PSD) if the loss function $\ell$ is convex.

In many problems, it is not easy to compute the Hessian matrix or even the GGN, but for deep learning problems there exists some good estimators in some cases.
Assume that the loss function is the negative log probability
associated with a distribution, that is, $\ell(f(x_i,\theta),y_i)= - \log p(y_i|x_i,\theta)$.
The corresponding FIM is defined as
\begin{equation*}
\begin{aligned}
\hspace{-0.5ex}\frac{1}{N} &\sum_{i=1}^N \mbE_{z\sim p(z|x_i,\theta)} \nabla_\theta  \ell(f(x_i,\theta),z)\pp{\nabla_\theta  \ell(f(x_i,\theta),z)}^\top.
\end{aligned}
\end{equation*}
For the square loss and the cross entropy loss function, GGN and FIM are equal. The empirical FIM (EFIM) matrix is also a good choice and is defined as:
 \[ \textbf{EFIM}  := \frac{1}{N} \sum_{i=1}^N \nabla \psi_i(\theta)\nabla \psi_i(\theta)^\top.\]
The evaluation of EFIM uses the sample gradients. Therefore, it does not require extra backward passes.
\section{Structured Stochastic Quasi-Newton Methods (\textbf{S2QN})} \label{sec:sto-qn}
In this section, we propose a structured stochastic quasi-Newton method to enhance curvature information. 
First, let us describe a second-order framework for the problem (\ref{finite-sum}). At the $k$-th iteration, a quadratic approximation model is constructed as follows:
\be
\label{quadratic-model}
\min_d \; m_k(d) := g_k^\top d + \frac{1}{2}d^\top (B_k+\lambda_k I) d,\ee
where $g_k$ and $B_k$ are the estimations of the gradient and the Hessian matrix at $\theta_k$, respectively. A common strategy to choose $g_k$ is the mini-batch gradient $ \nabla_{\mathcal{S}_g} \Psi(\theta)= \frac{1}{|\mathcal{S}_g|}
\sum_{i\in\mathcal{S}_g} \nabla \psi_i(\theta)$ for a given index set $\mathcal{S}_g \subseteq
 \{1,2,\dots,N\}$. The regularization parameter $\lambda_k$ is adjusted by the norm of the stochastic gradient. Specifically, for given $r_{1} <  r_{2}$ and sequence $\{\alpha_k\}$, we propose

\be
\label{lambda-adjust}
\lambda_{k} =
\begin{cases}
\frac{2r_1}{\|g_{k-1}\|+r_1}\alpha_k^{-1} &\|g_{k-1}\|<r_1, \\
\frac{2\|g_{k-1}\|}{\|g_{k-1}\|+r_2}\alpha_k^{-1} &\|g_{k-1}\|>r_2, \\
\alpha_k^{-1} &\text{otherwise}. \\
\end{cases}
\ee

 By solving the model (\ref{quadratic-model}), 
we update the parameter $\theta$ by 
$$\theta_{k+1} = \theta_k + \beta_k d_k, \quad d_k =- (B_k+\lambda_k I)^{-1}g_k,$$
where $\beta_k$ is the step size.
The strategy \eqref{lambda-adjust} is similar to (but still different from) the adjustment of the trust region radius in \cite{curtis2019fully}. It can be further viewed as a stochastic Levenberg-Marquardt method for generalized nonlinear least squares problem. The convergence analysis of our method is more straightforward compared to that in \cite{bergou2018stochastic}.
\subsection{Structured Stochastic quasi-Newton Matrix}
The key concept is to use a quasi-Newton method to compensate the difference between the partial Hessian information and the true Hessian matrix. That is, the
matrix $B_k$ is constructed as
\be \label{eq:Bk} B_k= H_k + \Lambda_{k}.\ee
Specifically, $H_k$ represents matrix with partial Hessian information which can be either the subsampled Hessian matrix, the GGN/FIM/EFIM, or any other approximation matrices.
The matrix $\Lambda_{k}$ compensates the difference with the true Hessian matrix. We call $H_k$ the base matrix and $\Lambda_k$ the refinement matrix. The quasi-Newton method is used to update $\Lambda_{k}$ under some constraints. For example,  $\Lambda_k = \text{QN}(u_{k-1},v_{k-1},\Lambda_{k-1})$ satisfies the secant condition: 
\be
\label{StructConI}
B_k u_{k-1} =  (H_k + \Lambda_k)u_{k-1} = \hat v_{k-1},
\ee
or equivalently,
\[
 \Lambda_{k} u_{k-1} =  \hat v_{k-1} -  H_k u_{k-1} := {v}_{k-1},
\]
where $u_{k-1} = \theta_{k} - \theta_{k-1}$ and $\hat v_{k-1} = \nabla_{\samples_{g}^{k-1}}\Psi(\theta_k)  - \nabla_{\samples_{g}^{k-1}}\Psi(\theta_{k-1}).$ Other structured conditions can be considered as long as they are reasonable. Since we compute $\hat v_{k-1}$ by using two gradients on the same samples in \eqref{StructConI} , we can consider an extra-step technique \cite{yangSEQN} if needed. Note that the pairs $\left<u_k,v_k\right>$ can be constructed by using other mechanisms, see \cite{BerNocTak16, ByrHanNocSin16, gower2016stochastic}. These methods can be integrated into our framework naturally.

\subsection{Representation of the Inverse Matrix}
 For an efficient computation of the direction $d_k$, we update the
 matrix $\Lambda_{k}$ by the limited-memory BFGS (L-BFGS) method
 \cite{nocedal2006numerical, sun2006optimization}. 
 Assume that there are $p$ pairs of vectors:
\be
\label{eq:uv-pair}
\begin{aligned}
U_k& = [u_{k-p},\dots,u_{k-1}]\in \mbR^{n\times p}, \\V_k& = [v_{k-p},\dots,v_{k-1}]\in\mbR^{n\times p}.
\end{aligned}
\ee
For a given initial matrix  $\Lambda_{k}^0$, a compact representation of the  L-BFGS matrix is:
\begin{equation} \label{eq:L-BFGS}
\Lambda_{k} =\text{QN}(U_k,{V}_k):=\text{LBFGS}(U_k,{V}_k)= \Lambda_{k}^0 - C_kP_k^{-1}C_k^\top,
\end{equation}
where
\begin{align*}
C_k:=&C_k(U_k,{V}_k)=\left[\begin{matrix}
\Lambda_{k}^0U_{k}, {V}_k
\end{matrix}\right]\in \mbR^{n\times 2p}, \\
P_k:=&P_k(U_k,{V}_k) =\left[ \begin{matrix}
U_k^\top \Lambda_{k}^0 U_k & L_k \\
L_k^\top & -D_k\\
\end{matrix} \right ]\in \mbR^{2p\times 2p},\\
(L_k)_{i,j}:=&(L_k(U_k,{V}_k))_{i,j}, \\
=&
\left \{
\begin{aligned}
&u_{k-p+i-1}^\top {v}_{k-p+j-1}    & \quad \text{if}\quad  i>j,\\
&0 &\quad \text{otherwise,}
\end{aligned} \right.
\\
D_k=&D_k(U_k,{V}_k) =\text{diag}\left[ u_{k-p}^\top {v}_{k-p},\dots,u_{k-1}^\top {v}_{k-1}\right].
\end{align*}
The initial matrix $\Lambda_{k}^0$ is usually set to be $ \gamma_k I$, where $\gamma_k$ is a positive scalar.
In order to ensure the positive definiteness of $\Lambda_{k}$, the pair $\{u_i,v_i\}$ should
satisfy $u_i^\top v_i \ge \epsilon_B \|u_i\| \|v_i\|$ with a small constant, say
$\epsilon_B=10^{-8}$. We can consider the damping strategy in \cite{WanMaGolLiu17}. Although $H_k$ may not be positive definite due to the nonconvexity
of the functions $\psi_i(\theta)$, the parameter
$\lambda_k$ can be adjusted suitably such that $B_k+\lambda_k I$ has good
properties to generate descent directions.

We now show how to compute the inverse of $B_k+\lambda_k I$. Let $\tilde{
H}_k=H_k+\Lambda_{k}^0+ \lambda_k I$. Assume that $\tilde H_k$ is invertible, otherwise the regularization parameter $\lambda_k$ can be adjusted accordingly. By using
the Sherman-Morrison-Woodbury (SMW) formula, we obtain
\begin{equation}
\label{inverse-hybrid-wk}
\begin{aligned}
(B_{k} + \lambda_k I)^{-1} 
&= (\tilde H_k - C_kP_k^{-1}C_k^\top)^{-1}  \\&= \tilde{H}_k^{-1} + \tilde{H}_k^{-1} C_k T_k^{-1}C_k^\top\tilde{H}_k^{-1},
\end{aligned}
\end{equation}
 where $T_k = P_k- C_k^\top\tilde{H}_k^{-1}C_k.$
Note that the main computational cost in \eqref{inverse-hybrid-wk} is the
inversion of $\tilde{H}_k$. Assume that $\Lambda_k^0$ is set to be $\gamma_k I$ and that $H_k$ can be easily obtained, then the computation of $\tilde{H}_k = H_k + (\gamma_k + \lambda_k)I$ is also cheap. Meanwhile, since $p$ is usually small, \eg, $1\sim 5$, the computational cost of \eqref{inverse-hybrid-wk} can be controlled.



\subsection{Explicit Inverse by Low-Rank Structures}

In many cases, the base matrix $H_k= Q_k Q_k^\top$ is low-rank, where $Q_k\in
\Rn^{n\times r}, r\ll n$. For example,  the subsampled EFIM in the sketchy natural gradient method \cite{yang2020SENG} is low-rank and its rank is related to the sample size.
For convenience of notation, we define:
\[\widetilde \Lambda_{k} =\Lambda_{k}^0+\lambda_k I, \; \widetilde P_k = \begin{bmatrix} P_k^{-1}&0 \\
0 & -I\end{bmatrix}, \; \widetilde C_k = [C_k, Q_k].\]
Using the SMW formula yields:
\be
\label{low-rank-inv}
\begin{aligned}
(B_k+\lambda_k I)^{-1} &=  (\widetilde \Lambda_{k}^0-\widetilde C_k \widetilde P_k^{-1}\widetilde C_k^\top)^{-1}\\
 &=  \widetilde \Lambda_{k}^{-1}+\widetilde \Lambda_{k}^{-1}\widetilde C_k\hat{T}_k^{-1} \widetilde C_k^\top\widetilde \Lambda_{k}^{-1},
 \end{aligned}
\ee
where $\hat{T}_k =
\widetilde P_k-
\widetilde C_k^\top\widetilde \Lambda_{k}^{-1}
\widetilde C_k$.
The size of $\hat{T}_k$ is $r + 2p$. Therefore, its inverse matrix can be computed fast.
Since $\Lambda_k^0$ is usually set to be $\gamma_k I$, the computational cost \eqref{low-rank-inv} can be much smaller than the cost of inverting the matrix directly.

  \subsection{Block Approximation}
Our structured quasi-Newton method can also be easily extended to the case when  $B_k =
\text{diag}\{B_k^1,\dots,B_k^L\}$ is a  block diagonal matrix. The case often occurs in practice. For $L$-layers neural network, to reduce the computational complexity, the curvature matrix is chosen to be a block-diagonal matrix and constructed layer by layer  \cite{grosse2016kronecker, yang2020SENG}. 

One requirement is 
that the $j$-th block of $\Lambda_{k}$ should satisfy the following secant equation:
$ \Lambda_{k}^j u_{k-1}^j = \hat v_{k-1}^j-H_k^ju_{k-1}^j := v_{k-1}^j$,
where $u_{k-1}^j$ and  $ v_{k-1}^j$ are the sub-vector of $u_{k-1}$ and
$v_{k-1}$ corresponding to the $j$-th block.

\LinesNumberedHidden
\begin{algorithm2e}[t]
\caption{Structured Stochastic Quasi-Newton Methods (S2QN)}
\label{S2QN}
\lnlset{alg:s2qn}{1}{Initialization: Choose an initial point $\theta_0$.
Select the sequence $(\alpha_k)$, $(\beta_k)$. Set the memory size $p$.} \\
\For{$k = 0,1,... $}{
\lnlset{alg:s2qn}{2}{Choose the random sample sets $\mathcal S_g^k, \mathcal S_H^k \subset [N]$. \\Compute $\nabla_{\samples_g^k} \Psi(\theta_{k})$ and base matrix $H_k$. \\
\lnlset{alg:s2qn}{3}{Compute the refinement matrix $\Lambda_{k} = \text{QN}(U_k,V_k)$.} \\
\lnlset{alg:s2qn}{4}{Adjust the regularization parameter $\lambda_k$ by \eqref{lambda-adjust} and compute the direction $d_k$. Update $ \theta_{k+1} = \theta_{k} + \beta_k d_k$. }} \\
\lnlset{alg:s2qn}{5}{Compute the gradient $ \nabla_{\samples_g^k} \Psi(\theta_{k+1})$ with the same samples and update the pairs $(U_{k+1},V_{k+1})$.} \\

}

\end{algorithm2e}
\section{S2QN For Deep Learning}
\label{S2QN-DL}
In this section, we extend the structured stochastic quasi-Newton method to deep learning problems by combining their characteristics. We first prove that the  block Hessian matrix for each convolutional layer is a kronecker product of two matrices under some assumptions, then propose a sketchy block BFGS method in this setting. 

The input of a convolutional layer is 	a set of activations $\{a_{j,t}\}$, where  $j \in \{1,2,\dots, \mathcal{J}\}$, $t\in \mathcal{T}$ and $\mathcal{T}$ is the set of spatial locations (typically a 2-D grid). We further write the effect area of the convolutional filter by $\Delta=\{-K,\dots,K\}\times \{-K,\dots,K\}$ and the learned weight by $\Theta$.
We also assume that the padding size equals $K$ and the stride is $1$. Assume that the outputs of the convolutional layer have $\mathcal{I}$ channels, then the set of pre-activations $\{s_{i,t}\}$ is as follows (we do not consider the bias for simplicity): $${s}_{i,t} = \sum_{\delta\in \Delta} \Theta_{i,j,\delta}a_{j,t+\delta},
$$
where  $i \in \{1,2,\dots, \mathcal{I}\}$ and $t\in \mathcal{T}.$ 
Therefore, the gradient of the component function with respect to one single sample $(x,y)$ can be computed as:
$\mathcal{D}\Theta_{i,j,\delta} = \sum_{t\in \mathcal{T}} a_{j,t+\delta} \mathcal{D}{s}_{i,t},$ or $\mathcal{D}\widetilde \Theta = G(x,y; \widetilde\Theta)A(x,\widetilde\Theta)^\top$ in matrix form
where we use the notation $\mathcal{D}(\cdot)$ to represent $\frac{\partial \psi}{\partial (\cdot)}$ and $\widetilde \Theta$ is the matrix form of $\Theta$. 
Differentiating again, we find that the elements of Hessian can be computed as:
\be
\label{CNN-Hessian}
	(\mathcal{D}^2\Theta)_{i,j,\delta;i', j',\delta'}
	=\sum_{t,t'\in \mathcal{T}}(\mathcal{D}^2{s})_{i,t;i',t'}a_{j, t+\delta}a_{ j',t'+\delta'},
\ee
where $(D^2 \Theta)_{i,j,\delta;i',j',\delta'} := \frac{\partial^2 \psi}{(\partial \Theta_{i,j,\delta})(\partial \Theta_{i',j',\delta'})}$ and $(D^2{s})_{i,t; i', t'} := \frac{\partial^2 \psi}{(\partial {s}_{i,t})(\partial {s}_{i',t'})}$.
However, since the size of the matrix is still large, it is not tractable to compute the block Hessian matrix. 
Instead, analogous to the approximation mechanism that KFAC made in \cite{grosse2016kronecker}, we approximate the Hessian matrix \eqref{CNN-Hessian} by a product of two smaller matrices. Similar results for fully-connected networks have been studied in \cite{practicalGN}. We list some necessary assumptions below.
\begin{assumption}
\label{kronecker-assump}
\begin{enumerate}
	\item [1.1)] The activations are independent of the second-order derivatives of the pre-activations.
	\item [1.2)]
	The second-order statistics of pre-activation derivatives at any two spatial locations $t$ and $t'$ depend only on $t'-t$, which means there exists function $\Gamma$ such that:
	\bee
		\mathbb{E}\left [(\mathcal{D}^2{s})_{t,i;t',i'} \right] =\Gamma(i,i',t'-t).
	\eee
\end{enumerate}
\end{assumption}
\begin{lemma}
\label{lemma-block}
Suppose that Assumptions \ref{kronecker-assump} are satisfied and there exists function $\Omega$ such that
\[\mathbb{E}\left [a_{j,t}a_{j',t'} \right]=\Omega(j,j',t'-t). \]
If the second-order derivatives of the pre-activations at any two distinct spatial locations are uncorrelated
  \bee
		\Gamma(i,i',t-t')=0, \quad\mbox{for} \quad t\neq t',
	\eee
	then we have:
\[
\mathbb{E}\left [(\mathcal{D}^2\Theta)_{i,j,\delta;i', j',\delta'}\right ] =\chi(\delta,\delta') \Omega(j,j',\delta'-\delta)\Gamma(i,i',0).
\]
Hence, the Hessian matrix is a kronecker product of two matrices:
\begin{equation}
\label{HessianKC}
\mathbb{E}\left [\frac{\partial^2 \psi}{\partial^2\ovec(\theta)}\right]= {\mathbf{A}}\otimes \mathbf{G},
\end{equation}
where 
$(\mathbf{A})_{j|\Delta|+\delta, j'|\Delta|+\delta'}=\chi(\delta,\delta')\Omega(j,j',\delta'-\delta)$, $(\mathbf{G})_{i,i'}=\Gamma(i,i',0),$
and $\chi(,)$ is a function for indexing the location. 
\end{lemma}

We omit the proof of  the Lemma \ref{lemma-block} since it is similar to Theorem 1 in \cite{grosse2016kronecker} which interested readers can refer to. By now, we have proved that the block Hessian matrix for convolutional layer can be written as a kronecker product of smaller matrices. The approximation is efficient. For example, for some layer in Resnet50 v1.5 with Imagenet-1k dataset, the size of block Hessian matrix is $2, 359, 296$ while the size of $\mathbf{G}$ and $\mathbf{A}$ are $512$ and $4, 608$, respectively. In the next, we show that by using the special properties of the block Hessian matrix, we can combine the structured quasi-Newton idea with KFAC approximation naturally.
\subsection{Stochastic Structured QN for Kronecker-Factored Approximation }
 The KFAC method \cite{grosse2016kronecker} approximates the FIM by a block-diagonal matrix where each block $\mfF_{\text{KFAC}}$ approximates the block Fisher matrix by a kronecker product of two smaller matrices as follows:
\be
\label{kfac-approx}
\mfF_{\text{KFAC}} =  \mathbf{ \hat A} \otimes \mathbf{\hat G},
\ee
where $\mathbf{ \hat  A} =  (\frac{1}{|\samples|}\sum_{x_i\in \samples} A(x_i,\Theta)A(x_i,\Theta)^\top)$ and 
$ \mathbf{ \hat G}= (\frac{1}{|\samples|}\sum_{x_i\in \samples} \mathbb{E}_{z \sim p(z|x_i,\theta)} G(x_i,z;\Theta)G(x_i,z;\Theta)^\top).$
However, extra backward passes are required to compute $\mathbf{ \hat G}$. Notice that \eqref{kfac-approx} has the similar kronecker structure as \eqref{HessianKC}. Hence, we can combine the structured QN idea and use the empirical version of \eqref{kfac-approx} as the base matrix to approximate the Hessian matrix \eqref{HessianKC}. Assume the refinement matrix $\Lambda_k$ takes the kronecker factorization, then the structured QN matrix is constructed as follows:
\be
\label{DL-Struc}
B_k = \mathbf{\hat{A}}_k \otimes \mathbf{\hat G}_k + \Lambda_k \approx \mathbf{\hat{A}}_k \otimes \left ( \mathbf{\tilde G}_k + \tilde \Lambda_k \right),
\ee
 where $\mathbf{\tilde G}_k  = \frac{1}{|\samples^k|}\sum_{(x_i,y_i) \in \samples^k}G(x_i,y_i;\Theta_k)G(x_i,y_i;\Theta_k)^\top.$

We propose two different conditions. 
The first strategy is to let $ \tilde \Lambda_k$ satisfy the stochastic secant equation: 
\begin{align*}
\left( \mathbf{\hat{A}}_k \otimes \left ( \mathbf{\tilde G}_k + \tilde \Lambda_k \right)\right) u_k = v_k,
\end{align*}
where $u_k =\theta_k - \theta_{k-1}$ and $ v_k =\nabla_{\samples_{g}^{k-1}}\Psi(\theta^k)  - \nabla_{\samples_{g}^{k-1}}\Psi(\theta^{k-1}).$
Equivalently, $\tilde \Lambda_k$ is required to satisfy:
\be \label{DL-cond1} \tilde \Lambda_k \hat U_k  = \hat V_k - \mathbf{\tilde G}_k\hat U_k \mathbf{\hat{A}}_k,\ee
where $\hat U_k$ and $\hat V_k$ are the matrix format of $u_k$ and $v_k$, that is, $u_k = \ovec(\hat U_k) , v_k = \ovec(\hat V_k)$. 

The second strategy is to make $ \mathbf{\tilde G}_k + \tilde \Lambda_k$ closer to $\mathbb{E}[(\mathcal{D}^2{s}_k)]$ directly.  In this case, $\tilde \Lambda_k$ should satisfy
\[\left( \mathbf{\tilde G}_k + \tilde \Lambda_k\right)\tilde U_k = \tilde V_k,\]
where $\tilde U_k = \text{mat} ({s}_{k} - {s}_{k-1})$, $\tilde V_k =\text{mat} ( \mathcal{D}{s}_{k} -  \mathcal{D}{s}_{k-1} )$. The operator $\text{mat}$ reshapes the arrays into matrices with correct sizes.  Equivalently, the condition is as follows:
\be \label{DL-cond2}\tilde \Lambda_k \tilde U_k = \tilde V_k - \mathbf{\tilde G}_k \tilde U_k.\ee
Under the spatial homogeneity assumptions, $\mathbb{E}[\mathcal{D}^2 {s}]$ is the same for all locations. Therefore, the above condition can be regarded as $|\mathcal{T}|$ secant conditions for all spatial coordinates. The condition \eqref{DL-cond2} can be equivalently formed as $\left(I \otimes (\mathbf{\tilde G}_k + \tilde \Lambda_k )\right)\ovec(\tilde U_k) = \ovec(\tilde V_k)$. 
Since $\mathbf{\hat{A}}_k$ is an estimator, the two conditions \eqref{DL-cond1} and \eqref{DL-cond2} are actually different. Note that both conditions are in matrix formats, and they can be viewed as modified multi-secant conditions. 


Recently, a practical QN method \cite{goldfarb2020practical} also takes the kronecker factorization \eqref{DL-Struc}. The distinction is that they construct two quasi-Newton schemes for the two parts of the kronecker product \eqref{HessianKC} for the fully-connected layer.  On the other hand, our proposed method compensates the base matrix with a quasi-Newton matrix and is also available for the convolutional layer. In the next, we show how to generate the refinement matrix by the sketchy block BFGS method.
\subsection{Sketchy Block BFGS Methods}
The refinement matrix $\Lambda_{k+1}$ in \eqref{DL-cond1} and \eqref{DL-cond2} takes the same format:
$
\Lambda_{k+1} \mathbb{U}_k = \mathbb{V}_k,
$
where $\mathbb{U}_k$ and $\mathbb{V}_k$ are matrices.
We use a block quasi-Newton update as follows:
\be\small
\label{block-bfgs-p}
\begin{aligned} 
	\Lambda_{k+1} &=\text{BlockQN}\left(\Lambda_{k},\mathbb{U}_k,\mathbb{V}_k,\mathbb{P}_k\right)\\
	&= \Lambda_k+\mathbb{V}_k(\mathbb{P}_k)^{-1}\mathbb{V}_k^\top -\Lambda_k\mathbb{U}_k(\mathbb{U}_k^\top \Lambda_k\mathbb{U}_k)^{-1}(\mathbb{U}_k)^\top \Lambda_k.
\end{aligned}
\ee
Let $\Lambda_{k+1}$ be symmetric and positive definite. However, since $\mathbb{V}_k^\top \mathbb{U}_k$ is not symmetric, $\mathbb{P}_k$ can be chosen as $\frac{1}{2} \left (\mathbb{V}_k^\top  \mathbb{U}_k + \mathbb{U}_k^\top  \mathbb{V}_k\right)$, $\text{Tr}(\mathbb{V}_k^\top  \mathbb{U}_k)$ or $\text{diag}(\mathbb{V}_k^\top  \mathbb{U}_k).$ We use the damping strategy to make $\mathbb{P}_k$ positive definite. Specifically, by choosing proper $\tau_k$, we replace $\mathbb{V}_k$ by $ \tau_k \mathbb{V}_k + (1-\tau_k) \Lambda_k \mathbb{{U}}_k$. Note that when $\mathbb{P}_k =\mathbb{V}_k^\top  \mathbb{U}_k,$ this is indeed the block BFGS method \cite{byrd1994representations} for multi-secant conditions. 


Since the size of $(\mathbb{V}_k)^\top \mathbb{U}_k$ and $\mathbb{U}_k^\top \Lambda_k \mathbb{U}_k$ is large, the computation of their inverse matrices is intractable. We use sketchy techniques to reduce the computational cost. Denote the sketching matrix by $\Xi_k$ and let $\mathbb{\tilde  V}_k =\Xi_k\mathbb{V}_k, \mathbb{\tilde  U}_k =\Xi_k\mathbb{U}_k$. Accordingly, the quasi-Newton update is changed to:
\be
\label{sketchy-BFGS}
	\Lambda_{k+1} = \text{BlockQN}\left(\Lambda_{k},\mathbb{\tilde U}_k,\mathbb{\tilde V}_k,\mathbb{\tilde P}_k\right),
\ee

where $\mathbb{P}_k$ is $\frac{1}{2} \left (\mathbb{\tilde V}_k^\top  \mathbb{\tilde U}_k + \mathbb{\tilde U}_k^\top  \mathbb{\tilde V}_k\right)$, $\text{Tr}(\mathbb{\tilde V}_k^\top  \mathbb{\tilde U}_k)$ or $\text{diag}(\mathbb{\tilde V}_k^\top  \mathbb{\tilde U}_k).$

Our method is different from the stochastic block BFGS studied in \cite{gower2016stochastic}. They design a special multi-secant condition by sketchy methods to squeeze more curvature information. In addition, $\mathbb{V}_k^\top \mathbb{U}_k$ is designed to be symmetric and positive definite in \cite{gower2016stochastic}.  However, our secant conditions \eqref{DL-cond1} and \eqref{DL-cond2} are constructed using the special kronecker structure and sketchy techniques are used to reduce the computational complexity.
\section{Theoretical Analysis} \label{sec:theo}
\subsection{Global Convergence}
In this part, we give a general  analysis for our structured stochastic quasi-Newton method under some mild assumptions.  When there is no $\Lambda_k$, the method degenerates to the Newton-type method. It is obvious that our
analysis fits both cases. Our global analysis is motivated by the strategies used in \cite{curtis2019fully,WanMaGolLiu17,yangSEQN}.

According to the stochasticity in the Algorithm \ref{S2QN}, we can define a filtration $\mathcal{F}_k$
such that $\theta_k \in \mathcal{F}_{k-1}$, $H_k$ and $\nabla_{\samples_g^k}\Psi(\theta_k)$ are $\cF_{k-1}$-independent.
A few necessary assumptions are listed below.
\begin{assumption} 
\begin{itemize}
\item[2.1)] $\Psi$ is continuously differentiable on $\Rn^n$ and is bounded from
    below by $\Psi_{\text{inf}}$. The gradient $\nabla \Psi$ is Lipschitz continuous on $\Rn^n$ with $L_\Psi \geq 1$.
\item[2.2)] For any iteration $k$, 
we assume that $B_k$ and $\nabla_{\samples_g^k}\Psi(\theta_k)$ are $\cF_{k-1}$ -independent and 
 it holds almost surely that
\begin{align*}
&\Expe[\nabla_{\samples_g^k} \Psi(\theta_k)|\mathcal{F}_{k-1} ] = \nabla \Psi(\theta^k).
\end{align*}
\item [2.3)] It holds almost surely that the variance of stochastic gradient is bounded
\[
 \Expe[\|\nabla_{\samples_g^k} \Psi(\theta_k) -\nabla \Psi(\theta_k)
 \|^2|\mathcal{F}_{k-1}]  \leq  \sigma_k^2.
\]
\item[2.4)] There exists positive constant $h$ such that for all $k$,
\[
0 \preceq B_k \preceq hI.
\]
\end{itemize}
\end{assumption}
The above assumptions are common and standard in stochastic QN methods
\cite{BerNocTak16,ByrHanNocSin16,gower2016stochastic,WanMaGolLiu17,yangSEQN}. As shown in Algorithm \ref{S2QN}, the base matrix $H_k$ is
$\mathcal{F}_{k-1}$-independent of $\nabla_{\samples_g^k}\Psi(\theta_k)$ and the refinement matrix
$\Lambda_k \in \mathcal{F}_{k-1}$. Hence, $B_k + \lambda_k I$ and  $\nabla_{\samples_g^k}\Psi(\theta_k)$ are $\mathcal{F}_{k-1}$-independent.

\begin{theorem}\label{thm:globalconv}
Suppose that Assumptions \ref{thm:globalconv} is satisfied, the step sizes $\beta_k\equiv 1$ and the sequence $\{\alpha_k\}_{k=1}^\infty$ satisfy:
\be
\alpha_k \leq \frac{r_1}{4r_2(L_\Psi+ h)}.
\ee
Then, under the conditions:
$
\sum \alpha_k = \infty, \;
\sum \alpha_k\sigma_k^2 < \infty,$
it holds for Algorithm \ref{S2QN}  almost surely that
\[
    \lim_{k \rightarrow \infty} \nabla \Psi(\te{k}) =0.
\]
\end{theorem}
The proof is shown in Appendix. In addition to the condition $\sum \alpha_k\sigma_k^2 < \infty,$ we can also consider the bounded variance assumption in \cite{WanMaGolLiu17} by adjusting relative conditions.
We next show the convergence results for a class of objective function that satisfies the Polyak- \L ojasiewicz (P\L) condition. The P\L  \ condition can bound the function value by the gradient norm squares and is widely used in algorithm analysis \cite{chang2018generalization, karimi2016linear, lei2019stochastic}. It is illustrated in \cite{foster2018uniform} that one-hidden neural networks and ResNets with linear activation functions satisfy the P\L \ condition. 
\begin{assumption}\textbf{(P\L \ condition)}
\label{PLcondition}
There exists a constant $c\in \left (0,\infty \right)$ for all $\theta \in \Rn^n$, such that
\[
2c (\Psi(\theta) - \Psi_{\text{inf}}) \leq \|\nabla \Psi(\theta)\|^2. 
\]
\end{assumption} 

\begin{theorem}\label{thm:globalconv}
Suppose that Assumptions \ref{thm:globalconv}-\ref{PLcondition} are satisfied and the step sizes $\{\beta_k\}_{k=1}^\infty$ and the sequence $\{\alpha_k\}_{k=1}^\infty$  satisfy:
\be
\alpha_k \equiv \alpha < \min \left \{ \frac{r_1}{4r_2(L_\Psi+ h)}, \frac{8}{c}  \right\}, \quad \beta_k\equiv 1.
\ee
In addition, assume that the variance of the stochastic gradient is decreased by a geometric speed, \ie,
$\sigma_k^2 \leq  M_\sigma  \zeta^{k}$
for some scalar $M_\sigma$ and $\zeta \in (0,1)$.
It holds for Algorithm \ref{S2QN}  almost surely that
\[
   \Expe[\Psi(\theta_k)] - \Psi_{\text{inf}} \leq \mu \nu^{k-1}, \]
 where $\nu = \max \{\zeta, 1-\frac{1}{16} c\alpha\} ,$ $\mu = \max\{\Psi (\theta_1) - \Psi_{\text{inf}}, \frac{15 M_{\sigma}}{c}\}.$
\end{theorem}
The proof is shown in Appendix.


\subsection{Local Convergence}
We analyze the convergence rate of the structured quasi-Newton method in a small local neighborhood of an optimal point. 
Consider the case where $\nabla^2 \Psi(\theta) = H(\theta) + \Pi(\theta)$ in \eqref{GGN-mtx}-\eqref{GGN2-mtx} and the sequence is updated as follows:
\be
\label{local-seq}
\theta_{k+1} = \theta_k - B_k^{-1}\nabla \Psi(\theta_k),
\ee 
where $B_k = H_{S_H^k}(\theta) + \Lambda_k$ and $H_{S_H^k} = \frac{1}{|S_H^k|} \sum_{i\in S_H^k}H_i(\theta)$. The refinement matrix
$\Lambda_k$ is generated by the BFGS method and satisfies :
$
\Lambda_k u_{k-1} = v_{k-1},
$
where $u_{k-1} = \theta_k - \theta_{k-1}$ and $v_{k-1}= \frac{1}{|\mathcal{S}_H^{k-1}|} \sum_{i\in \mathcal{S}_H^{k-1}}\left ( J_f^i(\theta_{k}) - J_f^{i}(\theta_{k-1})\right) \nabla_f\ell_i(\theta_{k})$. A few assumptions are listed below:

\begin{assumption}\label{local-assump}
\begin{enumerate}
\item [4.1)] The sequence $\{ \theta_k \}$ satisfies $\sum_k \|\theta_k - \theta^*\| < \infty$ a.s. for an optimal point $\theta^*$ where $\nabla^2 \Psi(\theta^*)$ are positive definite and there exists $\widetilde \lambda > 0$ such that for $i=1,\dots, n, $
	$
	\Pi_i(\theta^*) \succeq \widetilde \lambda I.
	$
\item [4.2)] 
The gradient $\nabla_f \ell_i(\theta)$ is bounded, the Hessian $\nabla^2_f \ell_i(\theta) $ is bounded and
 Lipschitz continuous near $\theta^*$ with Lipschitz constant $L_\ell$, $\forall i=1,\dots, N$, i.e., $
\|\nabla_f \ell_i(\theta) \| \leq \kappa_\ell,$
$\|\nabla^2_f \ell_i(\theta) \| \leq \widetilde \kappa_\ell$
 and $\|\nabla^2_f \ell_i(\theta_1) - \nabla^2_f \ell_i (\theta_1) \| \leq L_\ell \|\theta_1-\theta_2\|,$ for any $\theta_1, \theta_2$ near $\theta^*.$
\item [4.3)] The gradient $\nabla f_j^i(\theta)$ is bounded, the Hessian $\nabla^2 f_j^i(\theta)$ is bounded and Lipschitz continuous near $\theta^*$ with Lipschitz constant $L_f,$ $\forall i=1,\dots,N$ and $\forall j = 1,\dots,m$, i.e., $\|\nabla f_j^i(\theta)\| \leq \kappa_f$, $\|\nabla^2 f_j^i(\theta)\|\leq \tilde \kappa_f$ and $\|\nabla^2 f_j^i(\theta_1) - \nabla^2 f_j^i(\theta_2) \| \leq L_f \|\theta_1-\theta_2\|
$ for any $\theta_1, \theta_2$ near $\theta^*.$
\end{enumerate}
\end{assumption}
The above assumptions are common in the structured quasi-Newton method for deterministic nonlinear least squares problems \cite{zhou2010global}. To establish the fast local convergence results for stochastic methods, the sample size is required to increase superlinearly. These assumptions are similar to these in the subsampled Newton method  \cite{RooMah19} and stochastic BFGS method \cite{zhou2017stochastic}. We summarize the results as follows: 
\begin{theorem}
Suppose that Assumption \ref{local-assump} is satisfied. If the sample size $|S_H^k|$ increases superlinearly,
then the sequence $\{\theta_k\}$ generated by \eqref{local-seq} converges to $\theta^*$ superlinearly almost surely.
\end{theorem}
The proof is shown in Appendix. Note that we do not consider the stochastic mini-batch gradient in \eqref{local-seq}. However, similar results can be established by adding assumptions on the mini-batch gradient and its sample size.

\section{Numerical Experiments} \label{sec:num}
In this section, we compare our proposed method with a few standard methods for
logistic regression, deep autoencoders and convolutional neural networks. All methods
used in the experiments are briefly listed below and their detailed implementation is
reported in Appendix. 
\textbf{SGD} is the stochastic gradient method with momentum. \textbf{Adam} \cite{kingma2014adam} is an adaptive gradient method.
  \textbf{L-BFGS} \cite{nocedal2006numerical} is the well-known limited-memory quasi-Newton method.
\textbf{SSN} \cite{RooMah19} is the subsampled Newton method.
 \textbf{KFAC} \cite{grosse2016kronecker} is a method using the Fisher matrix for deep learning problems.
 \textbf{S4QN} is the variants of S2QN where the base matrix is
 subsampled Newton matrix. \textbf{SKQN-L}, \textbf{SKQN-B1} and \textbf{SKQN-B2} are the variants of S2QN when the base matrices are the empirical version of KFAC matrices but with different refinement matrix constructions \eqref{eq:L-BFGS}, \eqref{DL-cond1} and \eqref{DL-cond2}, respectively. 

\subsection{Logistic Regression}
In this part, we consider logistic regression problems for binary classification:
\begin{equation}
\label{eq:LR} \min_{\theta \in \mathbb{R}^n}  \Psi(\theta) = \frac{1}{N}\sum_{i=1}^{N} \log(1+\exp(-y_i \iprod{x_i}{\theta}) + \mu\|\theta\|^2,
\end{equation}
where $\{x_i,y_i\} \in \mathbb{R}^n \times \{-1,1\}$, $i\in [1, 2, \dots, N]$ correspond to a given dataset.
The statistics of the datasets used in our numerical comparisons are listed in
Appendix.

We compare  S4QN with SGD, SSN and L-BFGS.
Let one epoch be a full pass through the dataset and define the relative error as $\text{rel-err}:= (\Psi(\theta) - \Psi^{*})/\max\{1,\Psi^*\}$, where $ \Psi^{*}$ is the optimal function value. 
 The changes of the relative error with respect to the number of epochs for
 {rcv1} and {news20} is shown in
Figure \ref{LR:epochversusfunvalue1}.

We can observe that S4QN outperforms other methods greatly. It is worth
emphasizing that all settings of S4QN are the same as SSN except an additional quasi-Newton matrix.  The sample size of the gradient estimation in S4QN and SSN is increasing until the full batch. 
S4QN exhibits a faster local convergence rate than other methods.
These
facts illustrate that the structured methods can accelerate the Hessian-based and quasi-Newton  methods.

 \begin{figure}[h]
	\vskip -0.2in
\begin{center}
	\begin{tabular}{cc}
		\subfloat[rcv1]{
\includegraphics[width=0.22\textwidth]{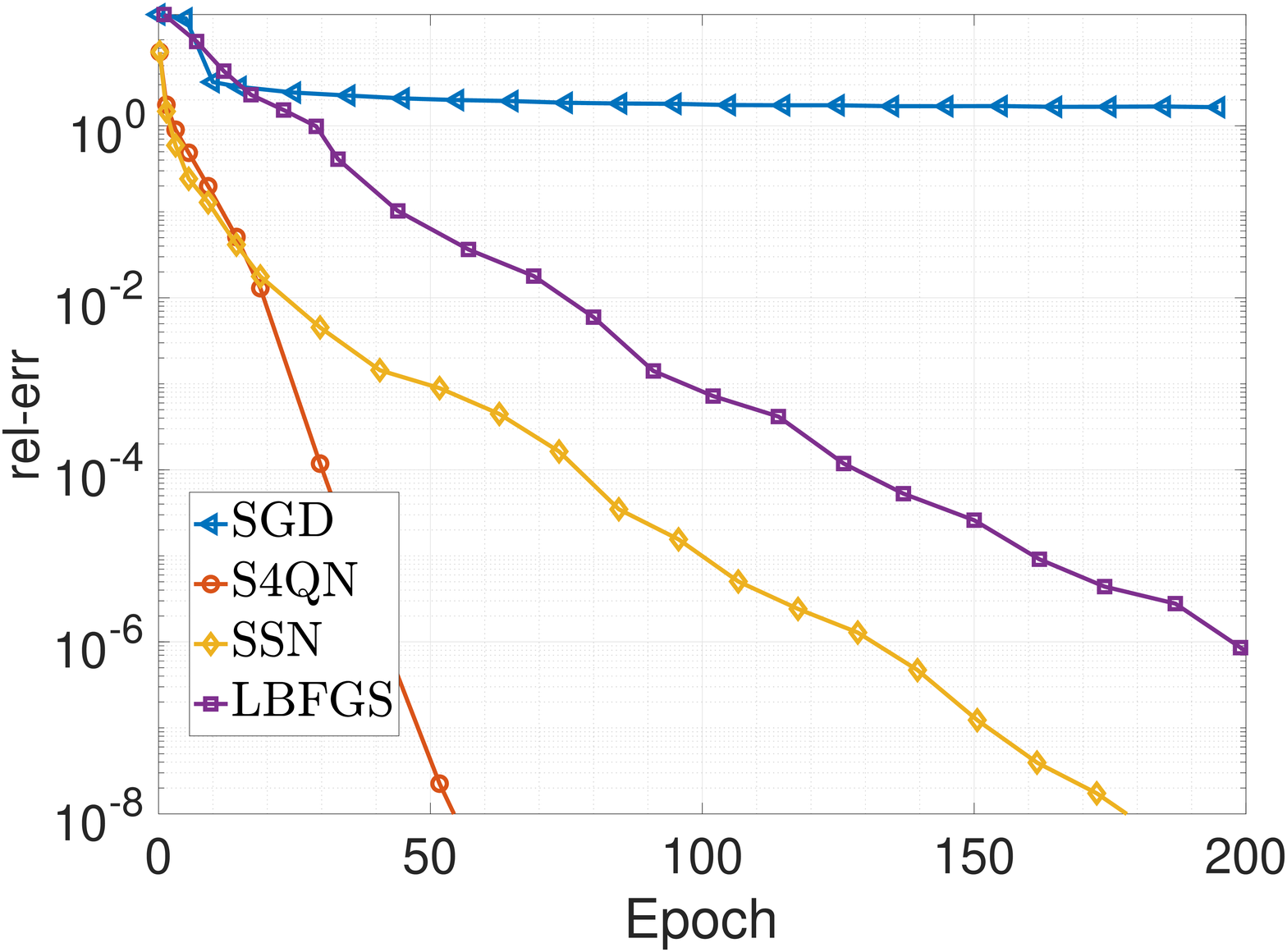}}
\subfloat[news20]{
\includegraphics[width=0.22\textwidth]{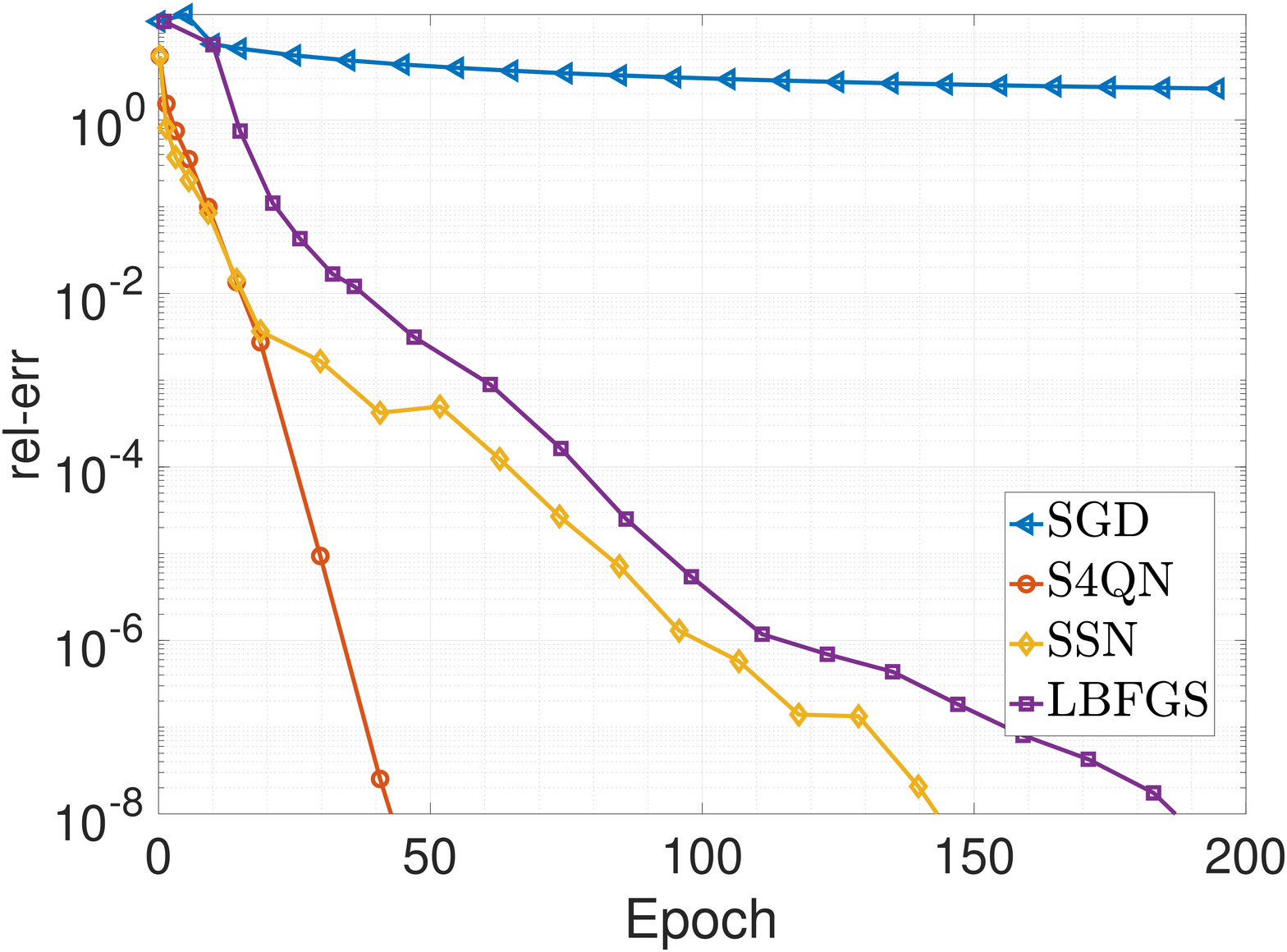}}\\
	\end{tabular}
\caption{Logistic Regression.}
\label{LR:epochversusfunvalue1}
\end{center}
\vskip -0.3in
\end{figure}

\subsection{Deep Autoencoders}
We next consider the deep autoencoder problem \cite{deep-auto-encoder} on three datasets: ``MNIST", ``CURVES" and ``FACES". The network architecture is $D$-1000-500-250-30-250-500-1000-$D$, where $D$ is the
dimension of the input data. We use the cross-entropy loss for CURVES and MNIST,
and the square error loss for FACES.
We compare SKQN-L, SKQN-B1and SKQN-B2 with SGD, ADAM and KFAC. We report the changes of the training loss and testing loss versus epochs in Figure \ref{AE:epochversusfunvalue}.

\begin{figure}[ht]
	\vskip -0.4in
	\begin{center}
		\begin{tabular}{cc}
		\subfloat[MNIST: Training Loss]{
\includegraphics[width=0.22\textwidth]{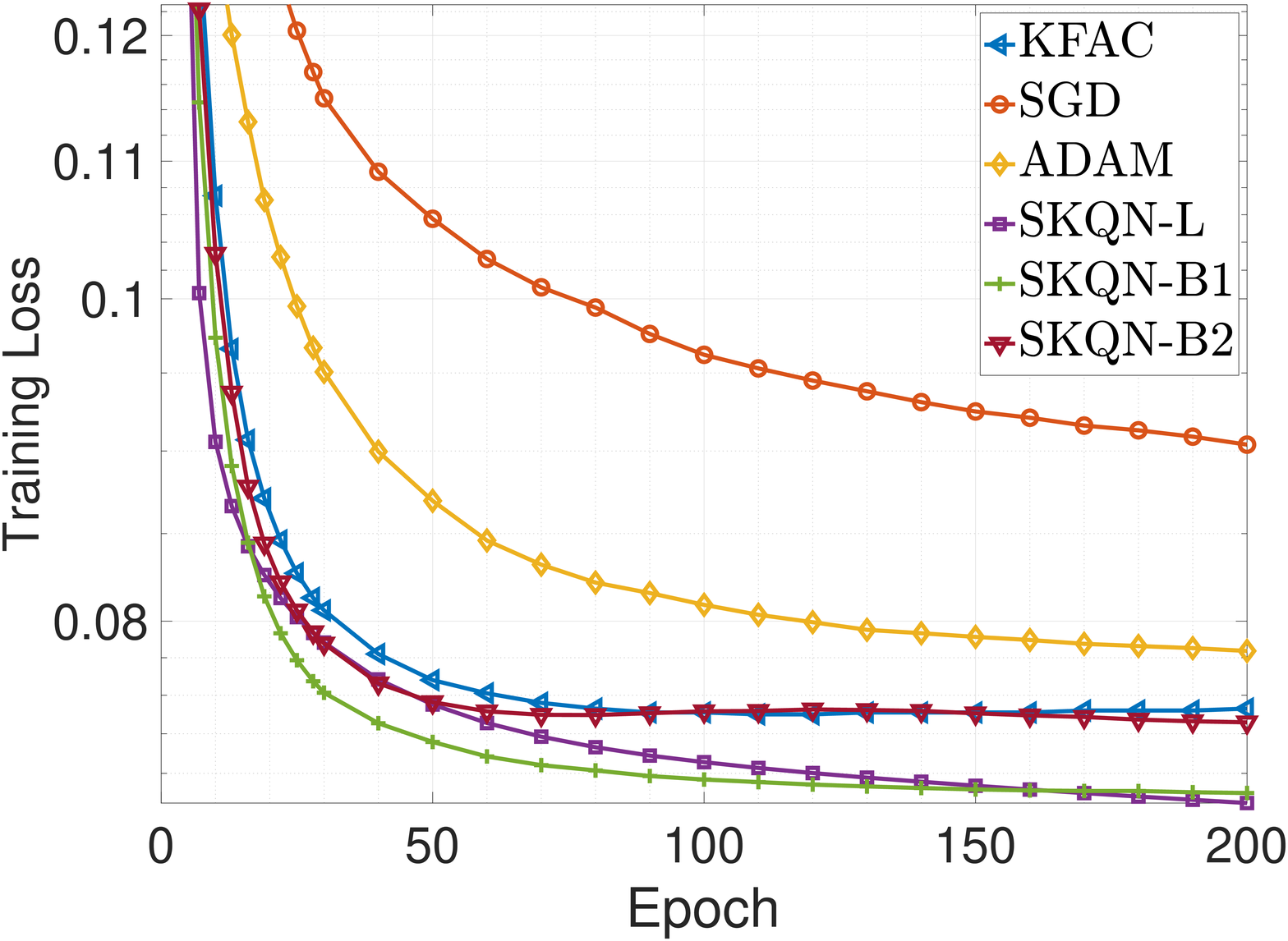}}
\subfloat[MNIST: Testing Loss]{
\includegraphics[width=0.22\textwidth]{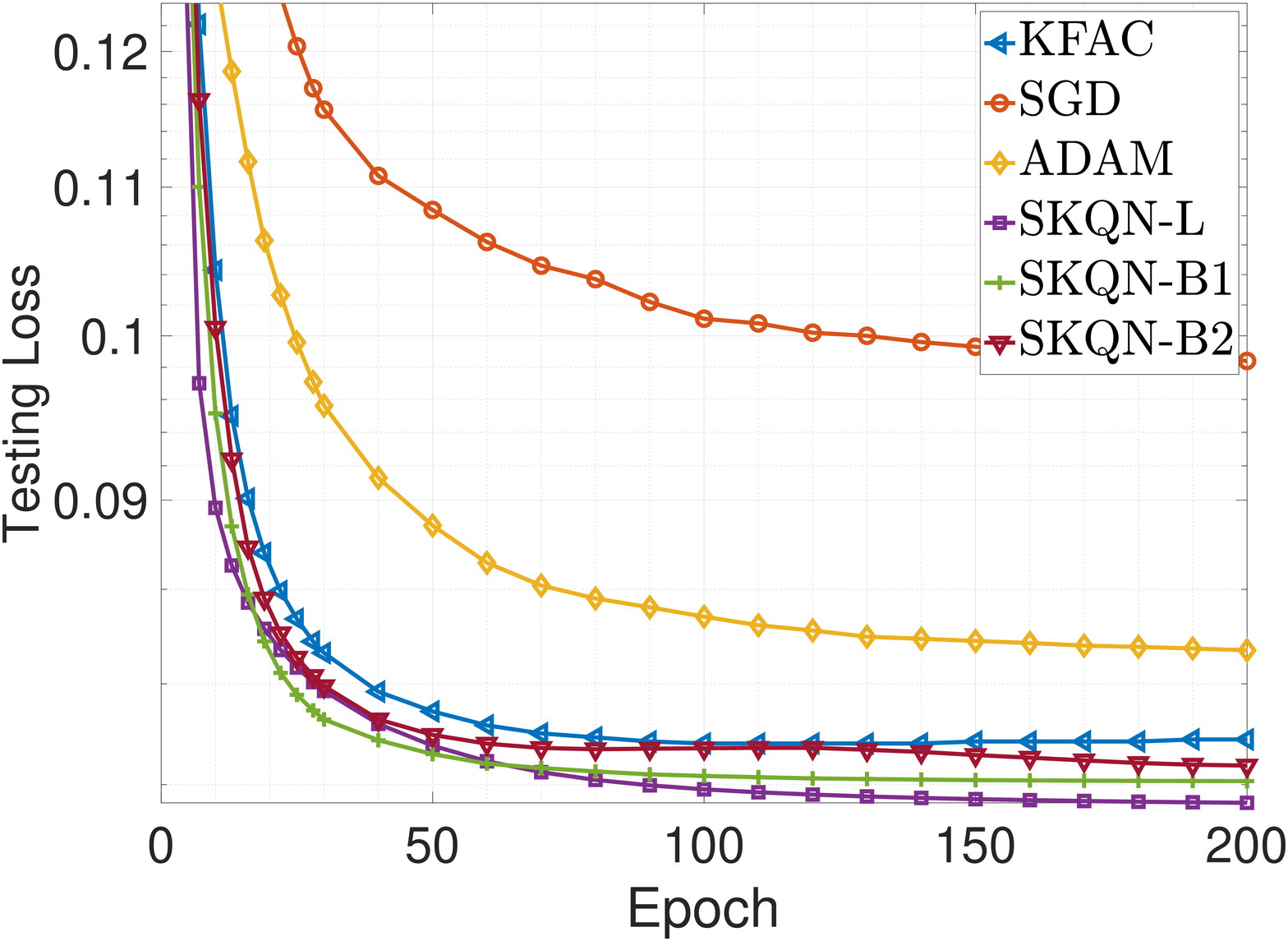}}\\
\subfloat[CURVES: Training Loss]{
\includegraphics[width=0.22\textwidth]{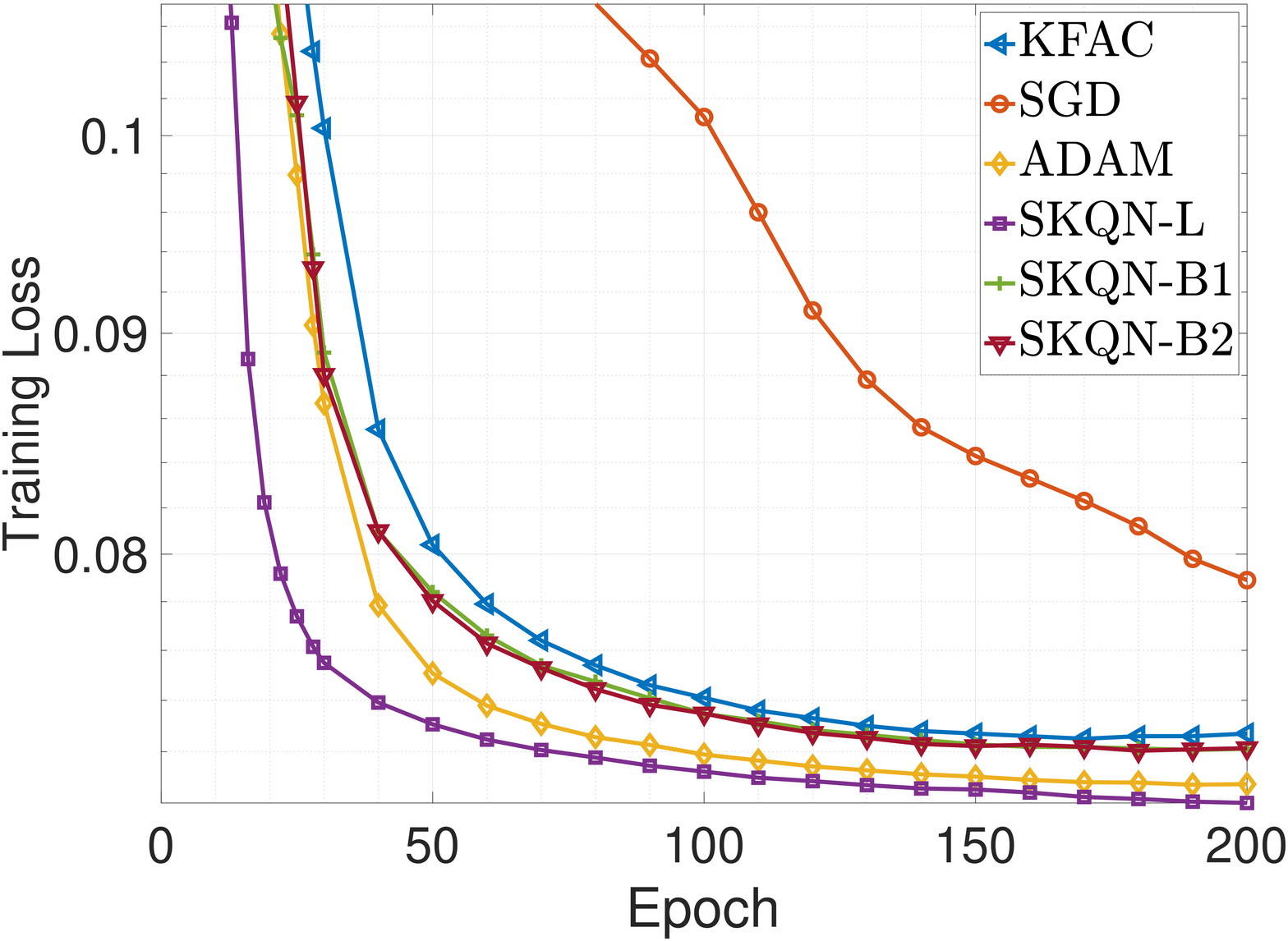}}
\subfloat[CURVES: Testing Loss]{
\includegraphics[width=0.22\textwidth]{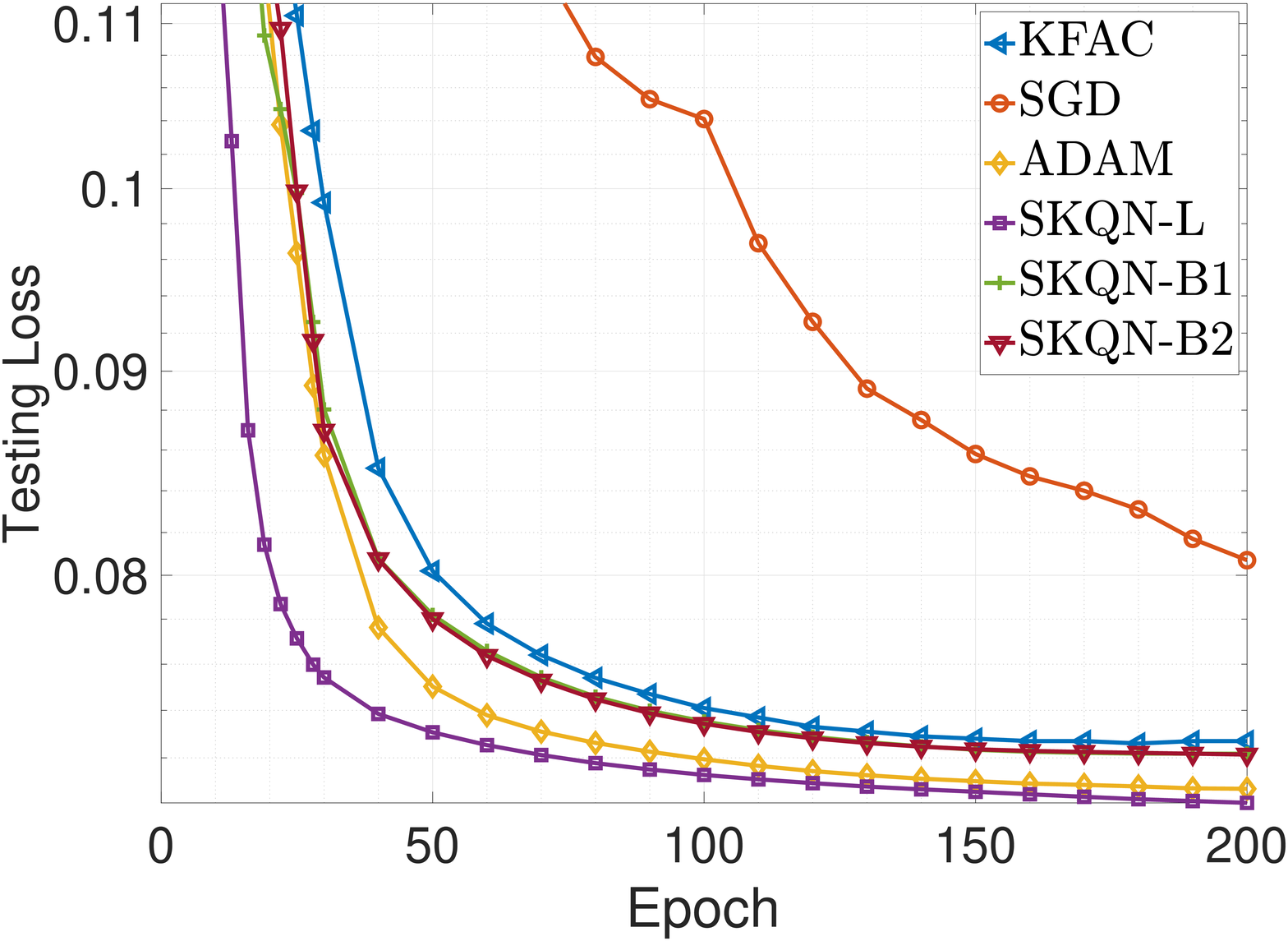}}\\
\subfloat[FACES: Training Loss]{
\includegraphics[width=0.22\textwidth]{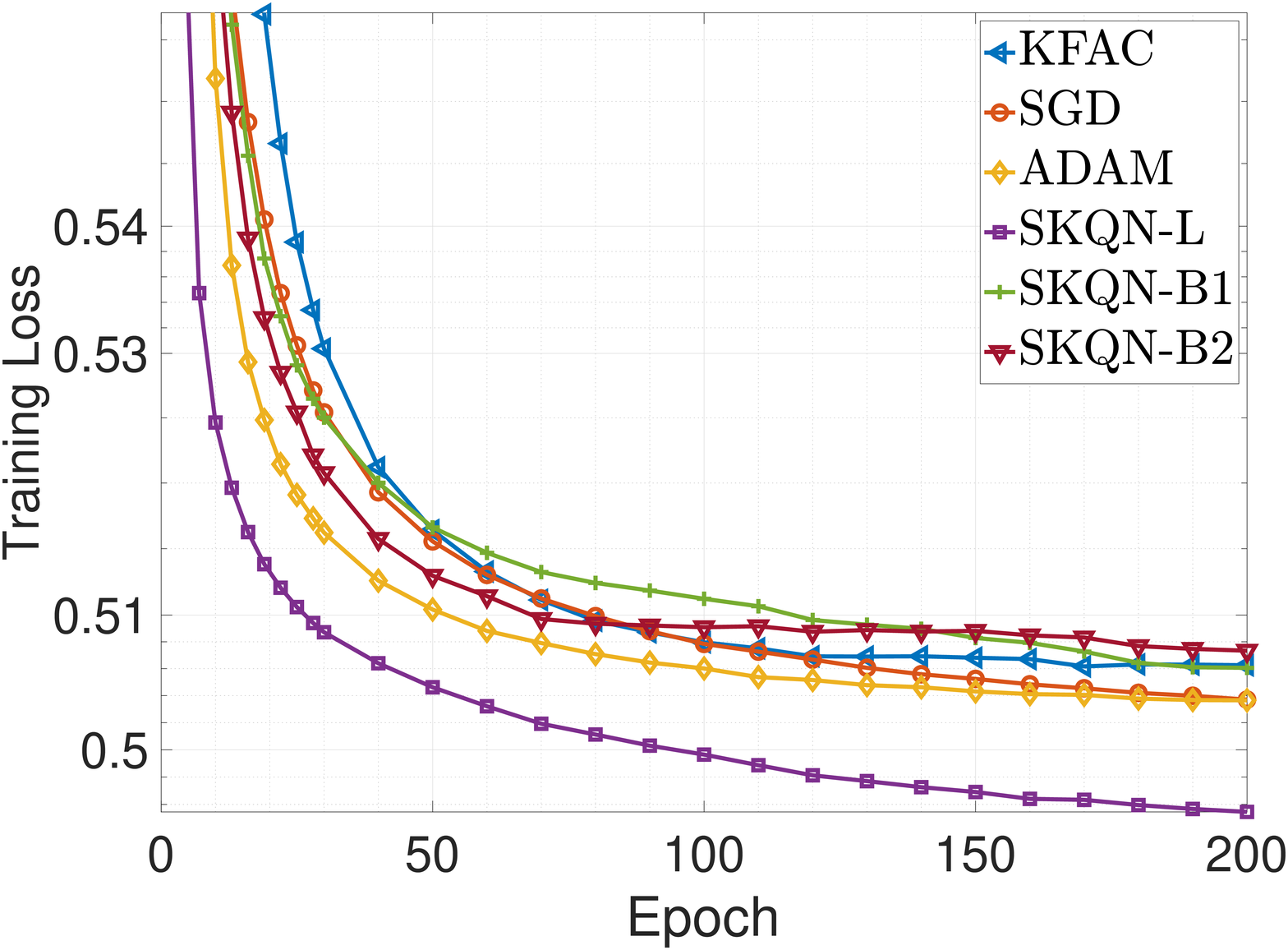}}
\subfloat[FACES: Testing Loss]{
\includegraphics[width=0.22\textwidth]{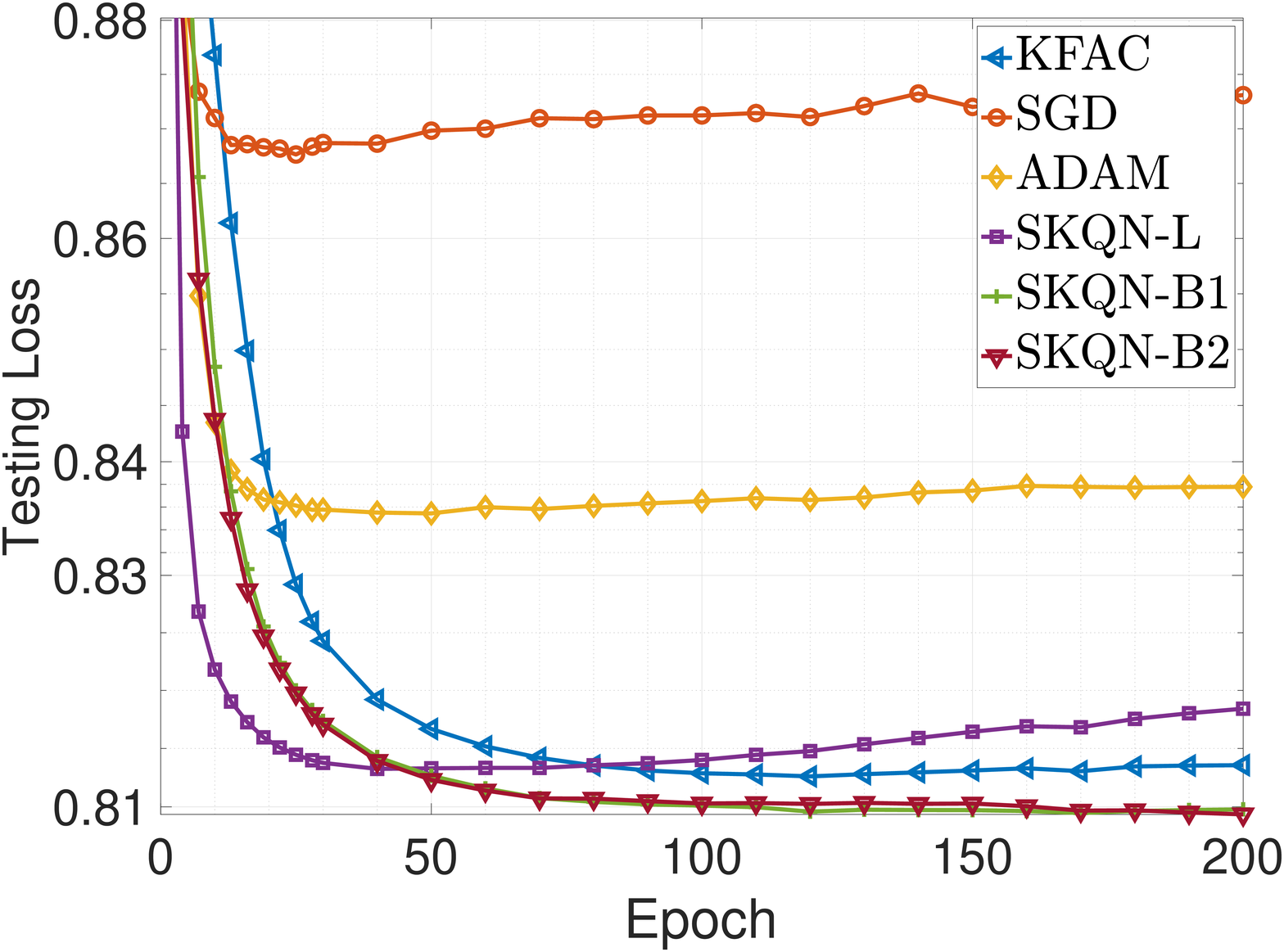}}
	\end{tabular}
    \caption{Autoencoders.}
	\label{AE:epochversusfunvalue}
	\end{center}
		\vskip -0.3in
\end{figure}

Compared to the first-order methods SGD and Adam, our structured quasi-Newton methods are better and
 more stable on all the three datasets. In comparison to KFAC, our proposed methods
 have improvements in both the training loss and testing loss while the computational cost does not
 increase  much at each iteration. Our results suggest that the structured
 quasi-Newton method with partial Hessian information indeed accelerates the
 convergence.

\subsection{ConvNet}
A 4-layer neural network ``ConvNet'' is considered in this part: three convolutional layers followed by a fully-connected layer. The detailed network architecture can be found in Appendix.
We compare algorithms with ``ConvNet'' on the CIFAR10 which is a standard dataset used for numerical performance comparison in deep learning.

The changes of the training loss and testing accuracy versus epochs are reported in Figure \ref{CIFAR10-ConvNet}. It is observed that the second-order type method is superior to the Adam and SGD. Our proposed methods have smaller training loss and larger testing accuracy in the end.
\begin{figure}[h]
\begin{center}
	\vskip -0.4in
	\begin{tabular}{cc}
		\subfloat[Training Loss]{
\includegraphics[width=0.22\textwidth]{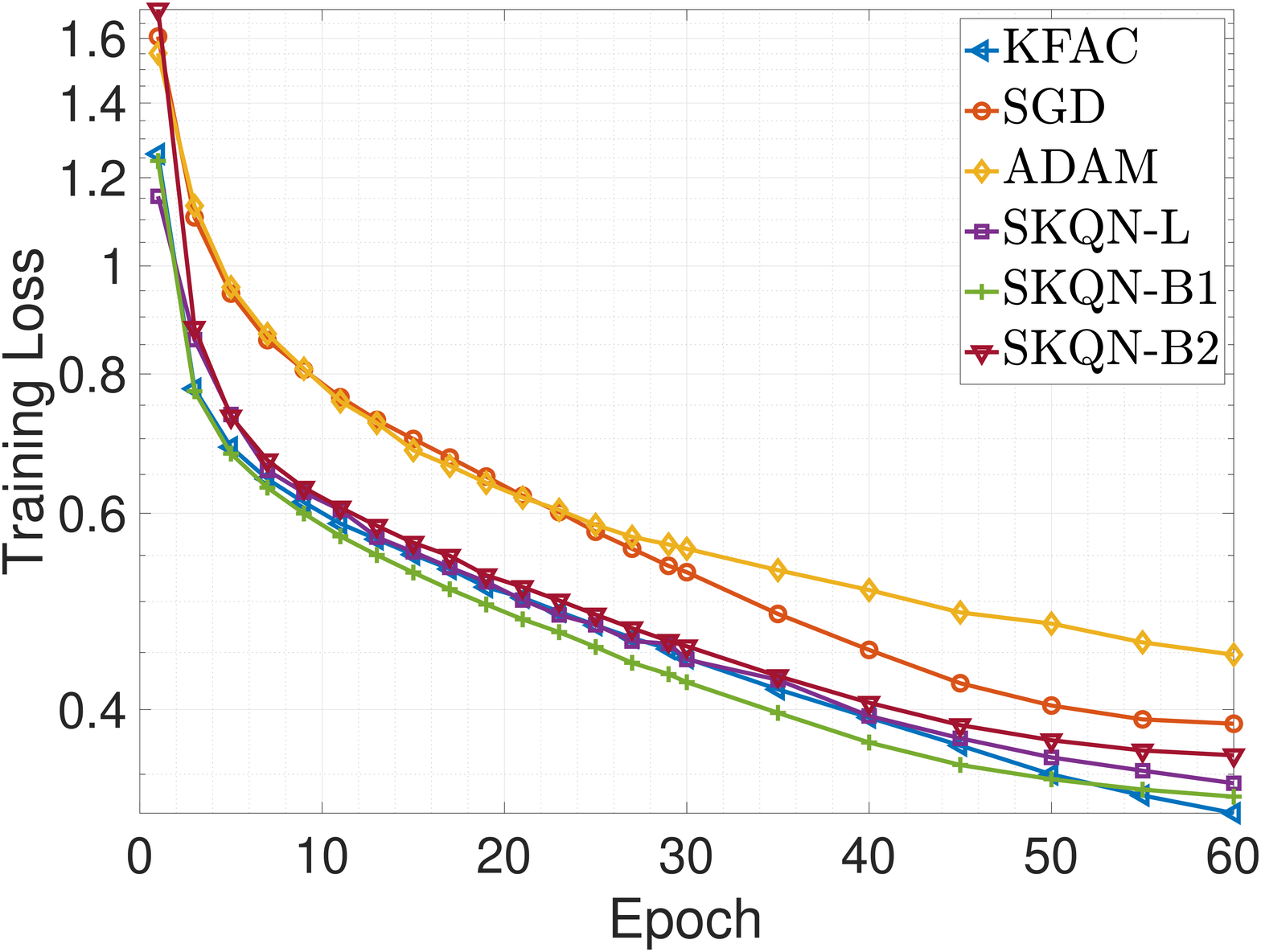}}
			\subfloat[Testing Accuracy]{
\includegraphics[width=0.22\textwidth]{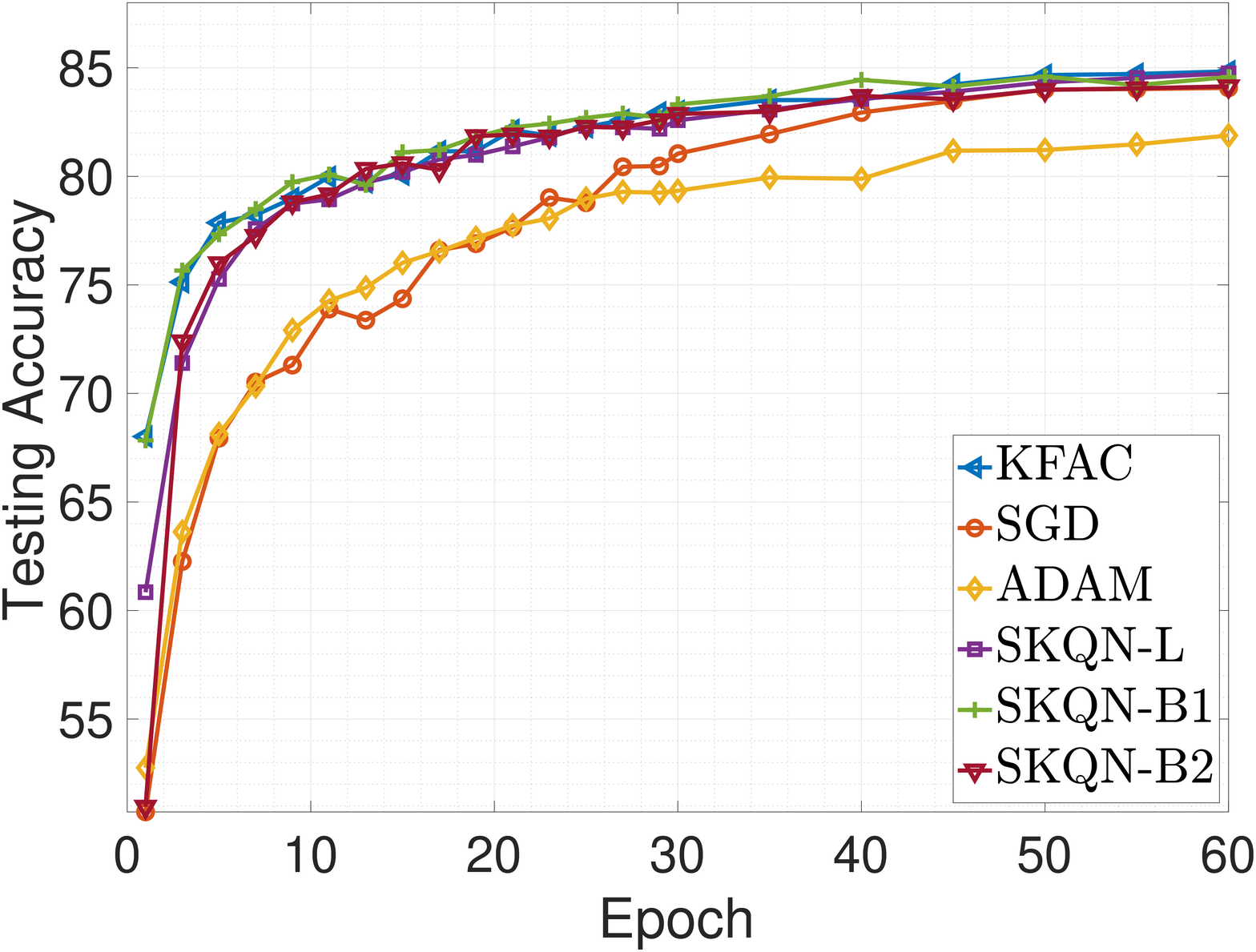}}\\
		\end{tabular}
	\caption{ConvNet on CIFAR-10.}
	\label{CIFAR10-ConvNet}
	\end{center}
		\vskip -0.3in
\end{figure}

\subsection{ResNet-18}
In this part, we consider the ``ResNet-18'' \cite{he2016deep} with the cross-entropy loss on the dataset
``CIFAR-10".  ResNet is a well-known network and widely used in practice.

The changes of the training loss and testing accuracy versus
epoch of CIFAR-10 are reported in Figure \ref{CIFAR10-deepCNN}. We can see that the second-order type
methods outperform the first-order type methods in terms of both criteria. The training error of our proposed quasi-Newton methods decreases
faster than that of KFAC. In terms of the testing accuracy, our proposed methods are better at start and at least comparable with KFAC in the end.
\begin{figure}[ht]
\begin{center}
   	\vskip -0.2in
	\begin{tabular}{cc}
		\subfloat[Training Loss]{
\includegraphics[width=0.22\textwidth]{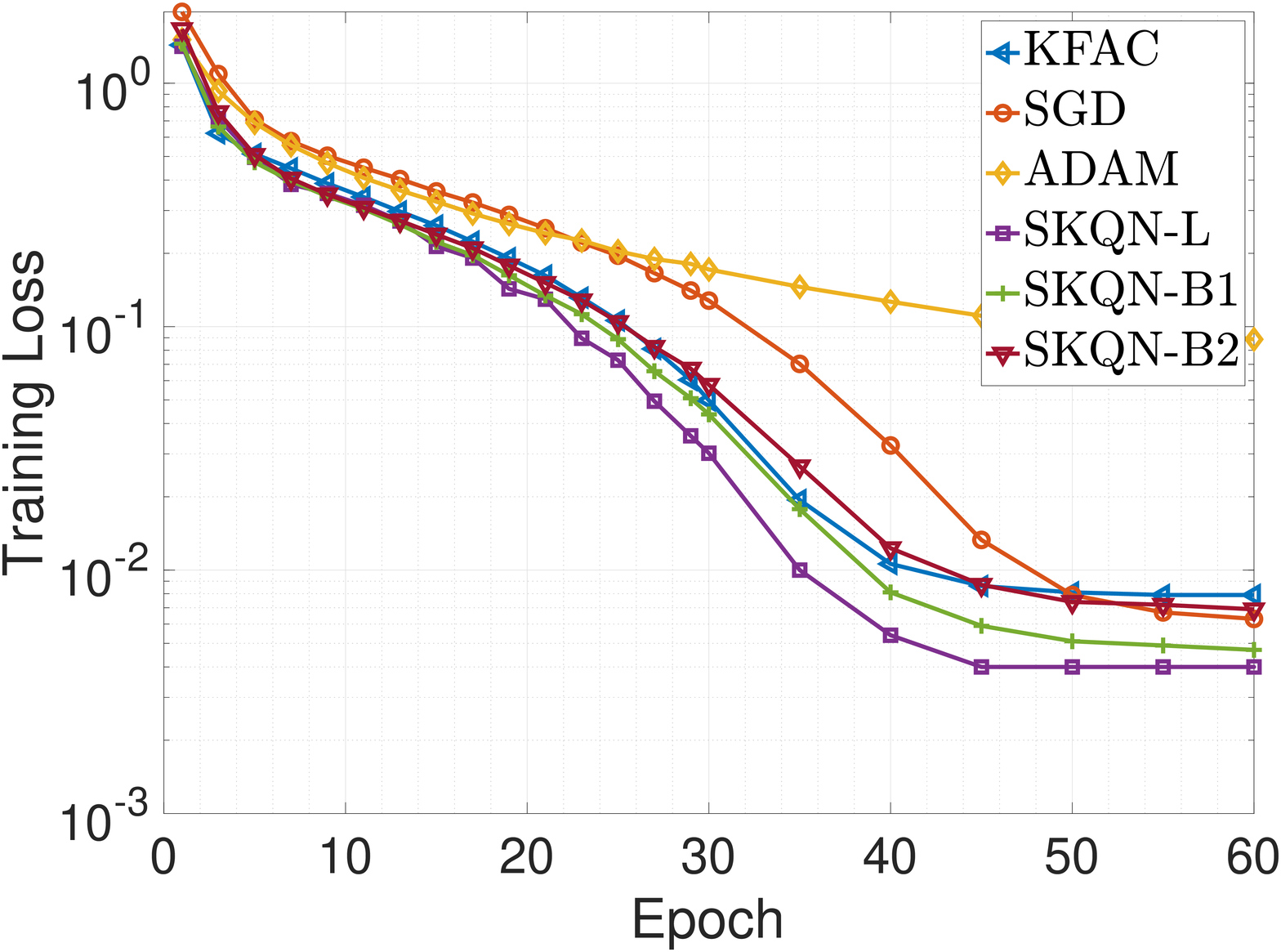}}
			\subfloat[Testing Accuracy]{
\includegraphics[width=0.22\textwidth]{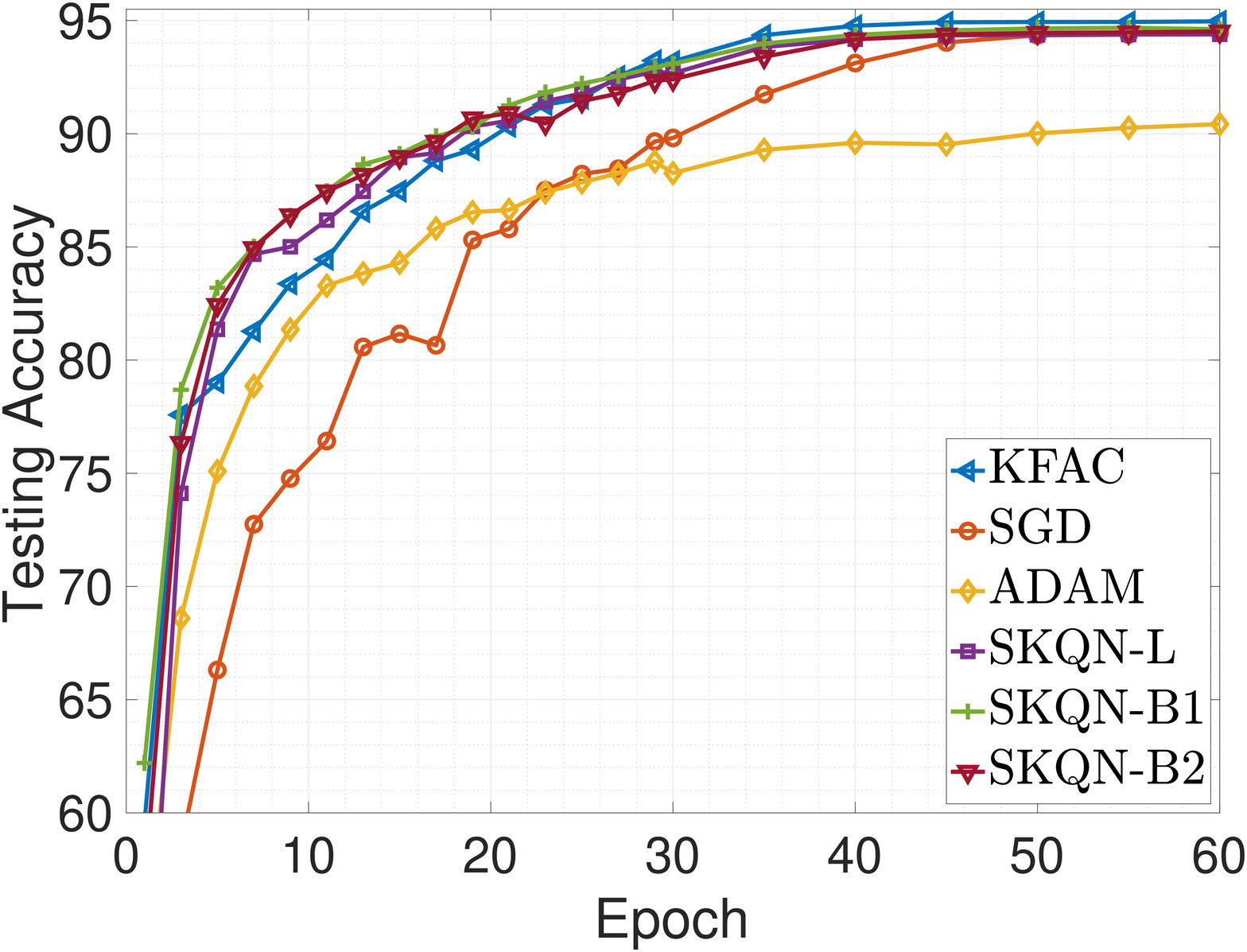}}
		\end{tabular}
	\caption{ResNet-18 on CIFAR-10.}
	\label{CIFAR10-deepCNN}
		\vskip -0.4in
	\end{center}
		\vskip -0.4in
\end{figure}
\section{Conclusion}
In this paper, a novel S2QN framework is proposed
and analyzed for large-scale finite-sum optimization problems.  Since the Hessian matrix can be split as a cheap part and an expensive part, we use the structured quasi-Newton method to exploit more curvature information for the expensive part. By further exploiting either the low-rank structure or the kronecker-product properties of the approximations, the computation of the quasi-Newton direction is affordable.
We obtain global convergence if the step sizes and the stochastic variances satisfy certain conditions. A local superlinear convergence result is also guaranteed under mind conditions. Our experimental results demonstrate the effectiveness of our structured stochastic quasi-Newton method compared to the state-of-the-art methods. In the future, we will implement our method on MindSpore\footnote{https://gitee.com/mindspore/}, a unified training and inference framework for device, edge and cloud in Huawei’s full-stack, all-scenario AI portfolio.

\textbf{Acknowledgments}  M. Yang, D. Xu and Z. Wen are supported in part by Key-Area Research and Development Program of Guangdong Province (No.2019B121204008),  the NSFC grants 11831002 and Beijing Academy of Artificial Intelligence.
{\small
\bibliographystyle{ieee_fullname}
\bibliography{ref}
}

\clearpage
\onecolumn
\begin{center}
\large{\bf \emph{Supplementary Material: \\Enhance Curvature Information by Structured Stochastic Quasi-Newton Methods}}
\end{center}
\section*{Implementation Details}
\subsection*{Logistic Regression}
\begin{itemize}
\item The objective function considered in this part is:
\bee
\label{eq:LR} \min_{\theta \in \mathbb{R}^n}  \Psi(\theta) = \frac{1}{N}\sum_{i=1}^{N} \log(1+\exp(-y_i \iprod{x_i}{\theta}) + \mu\|\theta\|_2^2,
\eee
where $\{x_i,y_i\} \in \mathbb{R}^n \times \{-1,1\}$, $i\in [1, 2, \dots, N].$ 
\item A description of the datasets is shown in Table \ref{LRdataset}.
\begin{table}[h]
\centering
\begin{tabular}{ccc}
\toprule
Dataset & \# Data points $N$ & \# Dimension  $n$ \\ 
\midrule
{rcv1} 	& 20, 242 		& 47, 236 \\ 
{news20} 	&19, 996		& 1, 355, 191 \\
\bottomrule
\end{tabular}
 \caption{A description of the datasets in logistic regression.}
\label{LRdataset}
\end{table}

\item We describe the implementation details of the algorithms used in this part.
\begin{itemize}
\item SGD: The batch size is set to be $1$.
\item {L-BFGS}: The source code is downloaded from the website $\footnote{\url{https://www.cs.ubc.ca/~schmidtm/Software/minFunc.html}}$ and the default parameters are used.
\item {SSN}: 
    The batch size $\samples_H$ for the subsampled Hessian matrix is
    $\min\{2000,\lfloor 0.01N \rfloor\}$. The batch size of the subsampled
    gradient $|\samples_g|$ is changing as
$ \min\{|\samples_g| \cdot1.1,N\}$.
\item {S4QN}: The set up of the subsampled Hessian $H_k$ is
    the same as SSN. The matrix $\Lambda_k$ is generated by the stochastic L-BFGS
    method and the memory size is $5$.
\end{itemize}
\end{itemize}

\subsection*{Deep Learning}
We now present the detailed implementation for the deep learning problems. The batch size for all methods is the same, i.e., $512$ for Autoencoders and $256$ for CNNs (ConvNet and ResNet-18). The hyper-parameters of Adam for all three architectures are tuned by using the grid search as follows.
\begin{itemize}
\item The initial learning rate is from \{3e-2, 1e-2, 3e-3, 1e-3, 3e-4, 1e-4\}. 
\item The parameters $\beta_1$ and $\beta_2$ are tuned in \{0.9,0.99\} and $\{0.99,0.999\}$, respectively. 
\item The perturbation value $\epsilon$ is 1e-8.
\end{itemize}
The hyper-parameters of other methods are tuned for their best numerical performance depending on the network architectures. We list the experimental settings and tuning mechanisms into two parts, Autoencoders and CNNs (including ConvNet and ResNet-18), respectively.

\subsubsection*{Autoencoders}
\begin{itemize}
\item Autoencoders are fully-connected neural networks. We test autoencoders on three datasets. The corresponding information is reported in Table \ref{AE:datasets}.

\item We describe the implementation details of the algorithms used in autoencoders.
\begin{itemize}
\item {SGD:} The stochastic gradient method with momentum $0.9$. The weight decay is set to be $10^{-5}$ and the learning rate is fixed to be the best one from $\eta_0 \in \{0.01, 0.02, 0.05, 0.1, 0.2, 0.5, 1, 2, 5\}$.
\item {KFAC:} The learning rate is set to $\eta =
\eta_0 \widehat \beta^{\mathrm{epoch}}$.
$\eta_0$ and $\widehat \beta$ is determined through grid search from $\eta_0 \in \{0.3, 0.5, 1\}$ and $\widehat \beta \in \{0.99, 1\}$.
The damping and the momentum parameter are set to be $0.2$ and $0.9$, respectively.
\item {SKQN-L:} The learning rate is set to $1$ in autoencoder for MNIST and FACES, $1.5$ for CURVES. The parameter $\gamma_k$ is set to $0.2 \times (\mathrm{epoch})^{0.99}$. The momentum is set to be 0.9 and the memory size is $5$. 
\item {SKQN-B1/SKQN-B2:} The learning rate is set to $0.7$ in autoencoder for MNIST, $0.4$ for FACES and $0.8$ for CURVES.  The damping is $\gamma_0 \times (\mathrm{epoch})^{0.99}$ with $\gamma_0=0.1$ for MNIST and CURVES, $0.2$ for FACES. The BFGS damping is set to be $0.5$ and the momentum is 0.9.
\end{itemize}
\end{itemize}
\begin{table}[t]
\centering
\begin{tabular}{ccccc}
\toprule
Dataset & \# Training & \# Testing  & Architecture & Loss \\ 
\midrule
MNIST 	& 60,000		& 10,000 & 784-1000-500-250-30-250-500-1000-784 &Cross-entropy \\ 
FACES	&103,500		& 62,100& 625-1000-500-250-30-250-500-1000-625&Mean squared error  \\
CURVES &20,000 & 10,000 &784-1000-500-250-30-250-500-1000-784&Cross-entropy \\
\bottomrule
\end{tabular}
 \caption{The corresponding information in autoencoders.}
\label{AE:datasets}
\end{table}
\subsection*{Deep CNNs}
In this part, we describe the implementation details for ConvNet and ResNet-18. The loss function is cross-entropy in these two problems. The hyper-parameters of each method are the same for both case unless otherwise specified.
\begin{figure}[t]
\begin{center}
   	\vskip -0.2in
	\begin{tabular}{cc}
		\subfloat[ConvNet]{
\includegraphics[width=0.2\textheight]{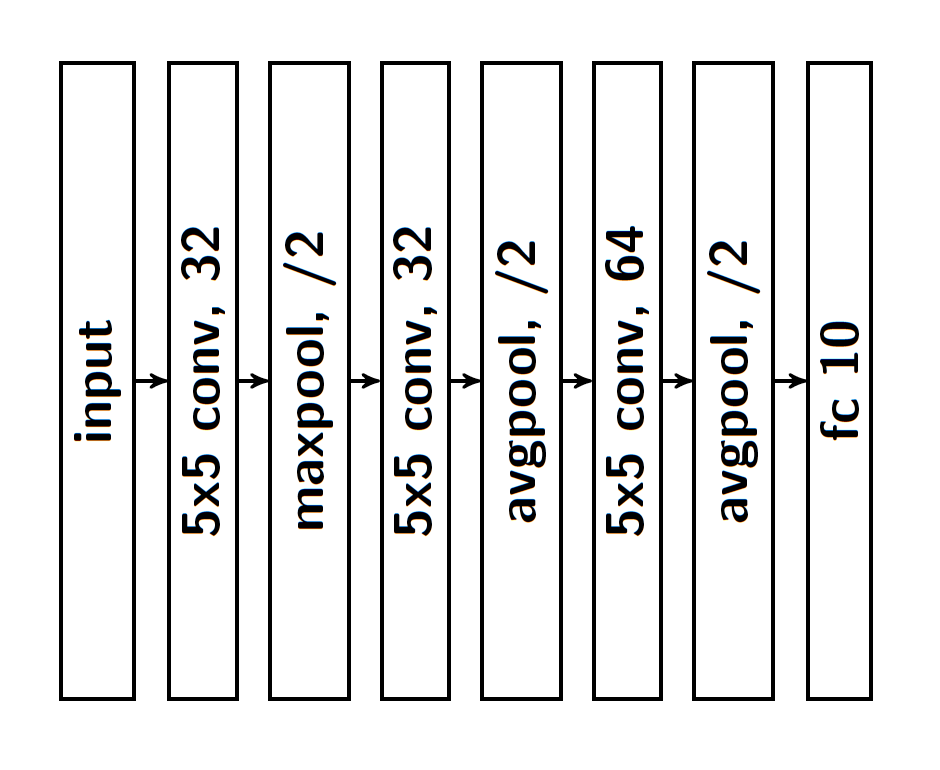}}\\
\subfloat[ResNet-18]{
\includegraphics[width=0.5\textheight]{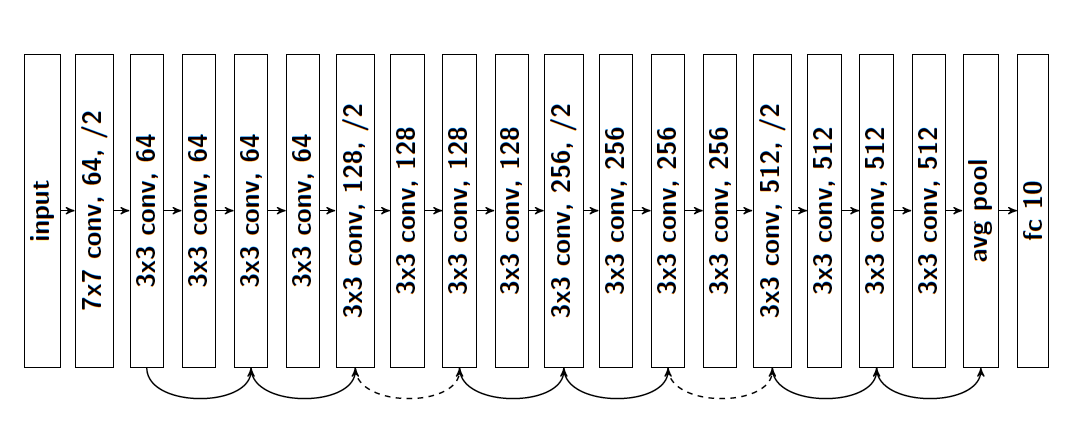}}
		\end{tabular}
	\caption{Network Architecture of ConvNet and ResNet-18.}
	\label{deepCNNs}
		\vskip -0.2in
	\end{center}
\end{figure}
\begin{itemize}
\item The network architectures used in ConvNet and ResNet-18 are presented in Figure \ref{deepCNNs}. ``conv'' in the figure means a sequence of convolutional kernel, Batch Normalization layer and Relu function. The numbers next to ``conv'' is the number of the channels of the outputs.
\item {SGD:} The momentum is set to be 0.9. The learning rate is is set to $\eta = \eta_0 (1 - \mathrm{epoch / epoch\_end})^{\widehat \beta}$. The parameters are determined by grid searching for the best result from $\alpha_0 \in \{0.01, 0.02, 0.05, 0.1, 0.2, 0.5, 1, 2, 5\}$, epoch\_end $\in \{80, 85, 90\}$ and $\widehat \beta \in \{4, 4.5, 5, 5.5, 6\}$. 
\item {KFAC:} The learning rate is $\eta = \eta_0 (1 - \mathrm{epoch / epoch\_end})^{\widehat \beta}$. The parameters are also determined from $\eta_0 \in \{0.01, 0.05, 0.1, 0.2, 0.5\}$, epoch\_end $\in \{70, 75, 80, 85\}$ and $\widehat \beta \in \{4, 5, 6\}.$ The damping parameter and the momentum parameter
are set to $0.7 \eta_0$, $0.9$, respectively. The curvature matrix is evaluated and inverted every
$50$ iterations.
\item {SKQN-L:} The memory size is $1$. The learning rate is set to be  $\eta = \eta_0 (1 - \mathrm{epoch / epoch\_end})^{\widehat \beta}$. We set $\eta_0 = 0.1$, epoch\_end = 85, $\widehat\beta = 4$ in the ConvNet and $\eta_0 = 0.15$, epoch\_end = 80, $\widehat\beta = 6$ in the ResNet-18, respectively. The damping is $0.7 \times \eta_0 (\eta / \eta_0)^{1/5}$. 
\item {SKQN-B1:} The learning rate for both cases is $\eta = 0.1\cdot (1 - \mathrm{epoch / 80})^5$. The damping is $0.8 \times 0.1\cdot (\eta / 0.1)^{1/5}$ in deep CNN problems. 

\item  {SKQN-B2:} The learning rate is set to be $\eta = \eta_0 (1 - \mathrm{epoch / epoch\_end})^{\widehat \beta}$ where we set $\eta_0 = 0.12$, epoch\_end = 85, $\widehat \beta = 5$ in the ConvNet and $\eta_0 = 0.1$, epoch\_end = 85, $\widehat \beta = 5$ in the ResNet-18. The damping  is $0.8 \times \eta_0 (\eta / \eta_0)^{1/5}$. 
\end{itemize}

\section*{A. Proof of Theorem 1}
\begin{proof}
It follows from Assumption 2.1) that the descent property holds:
\be
\label{descent:lemma}
\Psi(y) \leq \Psi(x) + \left< \nabla \Psi(x),y-x\right> + \frac{L_{\Psi}}{2}\|y-x\|^2.
\ee
Applying (\ref{descent:lemma}) and the Young inequality,  we obtain:
\be\label{descent-eq1}\begin{aligned}
&\Psi (\theta_{k+1}) - \Psi(\theta_k) \\
\leq
& \left <\nabla \Psi(\theta_k), \theta_{k+1} - \theta_k \right > + \frac{L_{\Psi}}{2}\|\theta_{k+1} - \theta_k\|^2\\
\leq & \left < \nabla_{\mathcal{S}^k_g} \Psi(\theta_k) , - (\lambda_k I + B_k)^{-1} \nabla_{\mathcal{S}^k_g} \Psi(\theta_k) \right > + \left < \nabla \Psi(\theta_k) - \nabla_{\mathcal{S}^k_g} \Psi(\theta_k)  , - (\lambda_k I + B_k)^{-1} \nabla_{\mathcal{S}^k_g} \Psi(\theta_k) \right > \\
& + \frac{L_{\Psi}}{2} \|(\lambda_k I + B_k )^{-1}\|_2^2 \|\nabla_{\mathcal{S}^k_g} \Psi(\theta_k) \|_2^2\\
\leq &  - (h + \lambda_k)^{-1}    \|\nabla_{\mathcal{S}^k_g} \Psi(\theta_k) \|_2^2 + 
\lambda_k^{-1}\| \nabla \Psi(\theta_k) - \nabla_{\mathcal{S}^k_g} \Psi(\theta_k) \|_2^2 +\frac{ \lambda_k}{4}\|  (\lambda_k I + B_k)^{-1}\|_2^2 \| \nabla_{\mathcal{S}^k_g} \Psi(\theta_k) \|_2^2 \\
& + \frac{L_{\Psi}}{2} \|(\lambda_k I + B_k )^{-1}\|_2^2 \|\nabla_{\mathcal{S}^k_g} \Psi(\theta_k) \|_2^2\\
\leq & - \left  [  (h + \lambda_k)^{-1}  - \frac{1}{4} \lambda_k^{-1} - \frac{L_{\Psi}}{2}\lambda_k^{-2}  \right] \|\nabla_{\mathcal{S}^k_g} \Psi(\theta_k) \|_2^2 + \lambda_k^{-1} \|\nabla \Psi(\theta_k) - \nabla_{\mathcal{S}^k_g} \Psi(\theta_k) \|^2.\\ 
\end{aligned}
\ee
Recalling that the parameter $\lambda_k$ is adjusted by the norm of the stochastic gradient as follows for a given $r_{1} < 1< r_{2}$:
\be
\label{lambda-adjust-supp}
\lambda_{k} =
\begin{cases}
\frac{2r_1}{\|g_{k-1}\|+r_1}\alpha_k^{-1} &\|g_{k-1}\|<r_1, \\
\frac{2\|g_{k-1}\|}{\|g_{k-1}\|+r_2}\alpha_k^{-1} &\|g_{k-1}\|>r_2, \\
\alpha_k^{-1} &\text{otherwise}, \\
\end{cases}
\ee
we prove that $  (h + \lambda_k)^{-1}  - \frac{1}{4} \lambda_k^{-1} - \frac{L_{\Psi}}{2}\lambda_k^{-2} $ is positive and bounded in all three cases.

We first consider the case when $\|g_{k-1}\|\in [r_1, r_2]$. Since $\lambda_k^{-1} = \alpha_k < \frac{r_1}{4r_2(L_{\Psi} + h)}$, we have $\frac{L_{\Psi}}{\lambda_k} < \frac{1}{4}$, and hence
\be
\label{estimateI}
\frac{1}{h + \lambda_k } - \frac{1}{4\lambda_k} - \frac{L_\Psi}{2\lambda_k^2}> \frac{1}{h + \lambda_k } - \frac{3}{8} \frac{1}{\lambda_k} > \frac{1}{8}\alpha_k.\ee
The last inequality follows from $\lambda_k = \alpha_k^{-1} > \frac{4 r_2(L_\Psi + h)}{ r_1} > h.$ As for the case when $\|g_{k-1}\|<r_1$, we have $$\lambda_{k}^{-1} = \alpha_k \frac{\|g_{k-1}\|+r_1}{2r_1}  \in [\frac12\alpha_k, \alpha_k]. $$ Then, we can still obtain
\be
\label{estimateII}
\frac{1}{h + \lambda_k } - \frac{1}{4\lambda_k} - \frac{L_\Psi}{2\lambda_k^2}> \frac{1}{h + \lambda_k } - \frac{3}{8} \frac{1}{\lambda_k} > \frac{1}{8}\frac{1}{\lambda_k} \geq \frac{1}{16}\alpha_k.\ee
For the last case when $\|g_{k-1}\|>r_2$, it follows $$\lambda_{k}^{-1} = \alpha_k \frac{\|g_{k-1}\|+r_2}{2\|g_{k-1}\|}  \in [\frac12\alpha_k, \alpha_k], $$ which implies the desired result as in (\ref{estimateII}).

Next, by using the Young inequality and taking conditional expectation based on $\cF_{k-1}$
 together with  $\Expe[\nabla_{S_g^k} \Psi(\theta^k)|\mathcal{F}_{k-1} ]
= \nabla \Psi(\theta^k)$ yields 
\be
\label{estimate-2}
\begin{aligned}
\Expe[\| \nabla_{\mathcal{S}^k_g}\Psi(\theta^k)\|^2|\cF_{k-1} ]
=& \Expe[\|\nabla_{\mathcal{S}^k_g} \Psi(\theta^k) - \nabla \Psi(\theta^k) + \nabla \Psi(\theta^k)\|^2|\cF_{k-1} ]\\
= &  \Expe [\|\nabla_{\mathcal{S}^k_g} \Psi(\theta^k) - \nabla \Psi(\theta^k) \|^2 |\cF_{k-1} ] +  \| \nabla \Psi(\theta^k)\|^2.
\end{aligned}
\ee
Taking the expectation related to $\samples_g^k$ of (\ref{descent-eq1}) on both
sides conditioned on $\cF_{k-1}$ and combining 
\eqref{descent-eq1}-\eqref{estimate-2}, we obtain
\be
\label{descent-eq2}
\begin{aligned}
&\Expe[\Psi(\theta_{k+1}) -   \Psi(\theta_{k})|\cF_{k-1}
]\\
\leq &
-\frac{1}{16}\alpha_k \|\nabla \Psi(\theta_{k})\|^2 + \left [\frac{1}{\lambda_k} -\frac{1}{16}\alpha_k \right ]\Expe [\|\nabla_{\mathcal{S}^k_g} \Psi(\theta^k) - \nabla \Psi(\theta^k) \|^2 |\cF_{k-1} ] \\ 
\leq &-\frac{1}{16}\alpha_k \|\nabla \Psi(\theta_{k})\|^2 + \tilde \beta_k \sigma_{k}^2,
\end{aligned}
\ee
where $\tilde \beta_k  = \frac{1}{\lambda_k} -\frac{1}{16}\alpha_k \leq(1 -\frac{1}{16})\alpha_k.$
 Taking expectation, summing over the inequality and using the
assumptions that there exists $\Psi_{\text{inf}}$ such that $\Psi(\theta) \geq \Psi_{\text{inf}}, \forall \theta \in \text{dom} \Psi $ , we obtain:
\be
\begin{aligned}
& \sum_{k=1}^\infty \frac{1}{16} \alpha_k \Expe\|\nabla \Psi(\theta_{k})\|^2
 \leq  \Psi(\theta_1) - \Psi^* +    \sum_{k=1}^\infty  \tilde \beta_k\sigma_k^2 .
 \end{aligned}
\ee
Therefore, we have  $\sum_{k=1}^\infty \alpha_k \Expe\|\nabla \Psi(\theta_k)\|^2< \infty $, which implies that $\sum_{k=1}^\infty \alpha_k \|\nabla \Psi(\theta_k)\|^2< \infty $ almost surely.
Consequently, we can infer \[
    \lim_{k\rightarrow \infty} \inf \nabla \Psi(\theta_{k}) = 0 \text{ almost surely }.
\]
Taking expectation, multiplying $\alpha_k$ on both sides of inequality \eqref{estimate-2} and summing over all $k$ yields
\bee
\begin{aligned}
&\sum_{k=1}^\infty \alpha_k\Expe\| \nabla_{\mathcal{S}^k_g}\Psi(\theta_k)\|^2
=  \sum_{k=1}^\infty \alpha_k \sigma_k^2 + \sum_{k=1}^\infty \alpha_k \| \nabla \Psi(\theta_k)\|^2 < \infty.
\end{aligned}
\eee
By the Young inequality, it implies
\beaa
    \sum_{k=1}^\infty\alpha_k^{-1}\Expe\|\theta_{k+1}-\theta_{k}\|^2 &=& \sum_{k=1}^\infty\alpha_k^{-1}\Expe\| (B_k+\lambda_k I)^{-1}\nabla_{\mathcal{S}^k_g}\Psi(\theta_{k})\|^2 \\
&\leq & \sum_{k=1}^\infty 2\frac{1}{\alpha_k(\lambda_k )^2 }\Expe\| \nabla_{\mathcal{S}^k_g}\Psi(\theta_{k})\|^2\\ &<& \infty.
\eeaa
It follows that
\[
\sum_{k=1}^\infty \alpha_k^{-1} \Expe\|\theta_{k+1} - \theta_k\|^2< \infty
\text{ and }
\sum_{k=1}^\infty \alpha_k^{-1} \|\theta_{k+1} - \theta_k\|^2< \infty \text{ almost surely}.
\]
On the  events $\mathcal{E} = \{\|\nabla \Psi(\theta_k)\|  \text{ does not converge}\}$, there exists $\epsilon>0$ and two increasing sequences $\{p_i\}_i $, $\{q_i\}_i$ such that $p_i < q_i$ and
\[
    \|\nabla \Psi(\theta_{p_i})\| \geq 2\epsilon, \quad \|\nabla \Psi(\theta_{q_i})\| < \epsilon, \quad \|\nabla \Psi(\theta_{k})\| \geq \epsilon,
\]
for $ k = p_i+1,\dots,q_i-1.$ Thus, it follows that
\be
\begin{aligned}
\epsilon^2 \sum_{i=0}^{\infty} \sum_{k=p_i}^{q_i-1} \alpha_k
& \leq \sum_{i=0}^{\infty} \sum_{k=p_i}^{q_i-1} \alpha_k \|\nabla \Psi(\theta_k)\|^2
\leq \sum_{k=0}^{\infty}\alpha_k \|\nabla \Psi(\theta_{k})\|^2< \infty .
\end{aligned}
\ee
Setting $\zeta_i = \sum_{k=p_i}^{q_i-1} \alpha_k$, it follows $\zeta_i\rightarrow 0.$ Then by the H\"older's inequality, we obtain
\[
    \|\theta_{p_i} -\theta_{q_i} \| \leq \sqrt{\zeta_i} [\sum_{k=p_i}^{q_i-1} \alpha_k^{-1}\|\theta_{k+1} - \theta_k\|^2]^{1/2} \rightarrow 0.
\]
Due to the Lipschitz property of $\nabla \Psi$, we have $\|\nabla \Psi(\theta_{p_i}) -\nabla \Psi(\theta_{q_i}) \|\rightarrow 0$, which is a contradiction. This implies $\mathbb{P}(\mathcal{E})= 0$. Hence, $\nabla \Psi(\theta_k) $ converges to zero almost surely.

\end{proof}

\section*{B. Proof of Theorem 2}
\begin{proof}
From the inequality \eqref{descent-eq2} in the proof of Theorem 1, we have 

\be
\begin{aligned}
&\Expe[\Psi(\theta_{k+1}) -   \Psi(\theta_{k})|\cF_{k-1}]
\leq &-\frac{1}{16}\alpha_k \|\nabla \Psi(\theta_{k})\|^2 +\frac{15}{16}\alpha_k \sigma_{k}^2.
\end{aligned}
\ee
Combining the Assumption 3.1) and $\sigma_k^2 \leq M_{\sigma}\zeta^{k-1}$, we have:
\[
\Expe[\Psi(\theta_{k+1}) -   \Psi_{\text{inf} }|\cF_{k-1}] \leq  (1-\frac{1}{8}c\alpha_k)\Expe[\Psi(\theta_{k}) -   \Psi_{\text{inf} }|\cF_{k-1}] + \frac{15}{16}\alpha_k M_{\sigma}\zeta^{k-1}.
\]
We prove Theorem 2 by induction. For $k=1$, the inequality holds by the definition $\mu = \max\{\Psi (\theta_1) - \Psi_{\text{inf}}, \frac{15 M_{\sigma}}{c}\}$. Then, we assume the inequality holds for $k\in \mathbb{N}$. Combining $\alpha_k \equiv \alpha < \min \left \{ \frac{r_1}{4r_2(L_\Psi+ h)}, \frac{8}{c}  \right\}$, $\mu = \max\{\Psi (\theta_1) - \Psi_{\text{inf}}, \frac{15 M_{\sigma}}{c}\}$ and $\nu = \max \{\zeta, 1-\frac{1}{16} c\alpha\}$, we have  
\be
\begin{aligned}
\Expe[\Psi(\theta_{k+1}) -   \Psi_{\text{inf} }|\cF_{k-1}] &\leq  (1-\frac{1}{8}c\alpha_k)\mu \nu^{k-1} +  \frac{15}{16} \alpha_k M_{\sigma}\zeta^{k-1} \\
&\leq \mu \nu^{k-1} \left(1-\frac{1}{8}c\alpha_k +  \frac{15}{16}M_{\sigma}\frac{\alpha_k}{\mu}       \right)\\
& \leq \mu \nu^{k-1}  \left(1-\frac{1}{16}c\alpha_k\right)\\
& \leq \mu \nu^{k},
\end{aligned}
\ee
which proves the inequality for $k+1$. This completes the proof of Theorem 2.

\end{proof}

\section*{C. Proof of Theorem 3}
We first state the settings of Theorem 3.
Consider the case when $\psi_{i}(\theta)=\ell_i(\theta) = \ell(f(x_i,\theta),y_i)$, where $f(\cdot,x): \Rn^{n } \rightarrow \Rn^{m} $. The Hessian matrix is $\nabla^2 \Psi(\theta) : = H(\theta) + \Pi(\theta)$, where
\begin{eqnarray}
    H(\theta)&=&\frac{1}{N} \sum_{i=1}^N H_i(\theta) = \frac{1}{N} \sum_{i=1}^N J_f^i(\theta)  \nabla_f^2 \ell_i(\theta)
\left(J_f^i(\theta)\right)^\top, \label{GGN-mtx-supp}\\
\Pi(\theta)&=&\frac{1}{N} \sum_{i=1}^N \Pi_i(\theta) =\frac{1}{N} \sum_{i=1}^N\sum_{j=1}^m \nabla_{f_j}
\ell_i(\theta)\nabla_\theta^2f_j^i(\theta), \label{GGN2-mtx-supp}
\end{eqnarray}
where $J_f^i(\theta)=\nabla_\theta f(x_i,\theta)\in \mbR^{n\times m}$ and
$f_j^i(\theta)$ is the $j$-th component of $f_i(\theta)$.

Consider the iteration in the neighborhood of $\theta^*$ : 
\be
\label{iteration-local-supp}
 \theta_{k+1} = \theta_k - B_k^{-1} \nabla \Psi(\theta_k).\ee

 Here we consider the true gradient for simplicity and the conclusion also holds for stochastic gradient by adjusting relative assumptions.

The curvature matrix is
\[
B_k =  H_{\mathcal{S}_H^k}(\theta) + \Lambda_k,
\]
where $ H_{\mathcal{S}_H^k}(\theta)$ is the base matrix and $\Lambda_k$ is the refinement matrix. We formulate $H_{\mathcal{S}_H^k}(\theta)$ as
\begin{eqnarray}
    H_{\mathcal{S}_H^k}(\theta)&=&\frac{1}{|\mathcal{S}_H^{k}|} \sum_{i\in \mathcal{S}_H^k}H_i(\theta), \label{GGN-mtx-supp}
\end{eqnarray}
and $\Lambda_k$ is generated by the BFGS method. Suppose that $\Lambda_k$ satisfies the following secant condition:
\be
\label{local-secant-supp}
\Lambda_k u_{k-1} = v_{k-1}, 
\ee

where $u_{k-1} = \theta_k - \theta_{k-1}$ and 
\be 
\begin{aligned} v_{k-1} =& \frac{1}{|\mathcal{S}_H^{k-1}|} \sum_{i\in \mathcal{S}_H^{k-1}}\left ( J_f^i(\theta_{k}) - J_f^{i}(\theta_{k-1})\right)\nabla_f\ell_i(\theta_{k})\\
 =& \frac{1}{|\mathcal{S}_H^{k-1}|} \sum_{i\in \mathcal{S}_H^{k-1}} \sum_{j=1}^m \left ( \nabla_{\theta} f_j^i(\theta_{k}) -\nabla_{\theta} f_j^i(\theta_{k-1} )\right) \nabla_{f_j}\ell_i(\theta_k).
\end{aligned}
\ee

We want to prove that if the sample size is sufficiently large, then the stochastic Dennis-More condition holds. Hence, the local superlinear convergence speed can be guaranteed.
A few assumptions are listed below.
%
%
%
\begin{assumption}\label{local-assump-supp}
\begin{enumerate}
\item [1.1)] The sequence $\{ \theta_k \}$ satisfies $\sum_k \|\theta_k - \theta^*\| < \infty$ a.s. for an optimal point $\theta^*$ where $\nabla^2 \Psi(\theta^*)$ are positive definite and there exists $\widetilde \lambda > 0$ such that for $i=1,\dots, n, $
	$
	\Pi_i(\theta^*) \succeq \widetilde \lambda I.
	$
\item [1.2)] 
The gradient $\nabla_f \ell_i(\theta)$ is bounded, the Hessian $\nabla^2_f \ell_i(\theta) $ is bounded and
 Lipschitz continuous near $\theta^*$ with Lipschitz constant $L_\ell$, $\forall i=1,\dots, N$, i.e., $
\|\nabla_f \ell_i(\theta) \|_2 \leq \kappa_\ell,$
$\|\nabla^2_f \ell_i(\theta) \|_2 \leq \widetilde \kappa_\ell$
 and $\|\nabla^2_f \ell_i(\theta_1) - \nabla^2_f \ell_i (\theta_1) \|_2 \leq L_\ell \|\theta_1-\theta_2\|,$ for any $\theta_1, \theta_2$ near $\theta^*.$
\item [1.3)] The gradient $\nabla f_j^i(\theta)$ is bounded, the Hessian $\nabla^2 f_j^i(\theta)$ is bounded and Lipschitz continuous near $\theta^*$ with Lipschitz constant $L_f,$ $\forall i=1,\dots,N$ and $\forall j = 1,\dots,m$, i.e., $\|\nabla f_j^i(\theta)\|_2 \leq \kappa_f$, $\|\nabla^2 f_j^i(\theta)\|_2\leq \tilde \kappa_f$ and $\|\nabla^2 f_j^i(\theta_1) - \nabla^2 f_j^i(\theta_2) \|_2 \leq L_f \|\theta_1-\theta_2\|_2
$ for any $\theta_1, \theta_2$ near $\theta^*.$
\end{enumerate}
\end{assumption}

\begin{lemma}\label{local-lemma}
Under Assumption \ref{local-assump-supp}, the following conclusions hold:
\begin{itemize}
\item The GGN matrix $H_i(\theta)$ and $\Pi_i(\theta)$ are bounded for $i=1,\dots,n,$ i.e., there exists constants $\kappa_H$ and $\kappa_\Pi$ such that for all $\theta$ near $\theta^*$,
	\[ \|H_i(\theta)\|\leq \kappa_H \quad  \text{and}  \quad \|\Pi_i(\theta)\| \leq \kappa_\Pi . \]
\item  $J_f^i(\theta)$ is Lipschitz continuous near $\theta^*$ with Lipschitz constant $L_J$, $\forall i=1,\dots, N$, i.e.,
\[
\|J_f^i(\theta_1) - J_f^i(\theta_2) \| \leq L_J \|\theta_1-\theta_2\|,
\]
for any $\theta_1, \theta_2$ near $\theta^*$.

\item$H(\theta)$ is Lipschitz continuous near $\theta^*$ with Lipschitz constant $L_H$, $\forall i=1,\dots, N$, i.e., 
	\[\|H(\theta_1) - H(\theta_2)\|\leq L_H \|\theta_1-\theta_2\|,\]
	for any $\theta_1, \theta_2$ near $\theta^*$.

\end{itemize}
\end{lemma}
\begin{proof}
We first prove that $H_i(\theta)$ and $\Pi_i(\theta)$ are bounded. 
Recalling $H_i(\theta)= J_f^i(\theta) \nabla_f^2 \ell_i(\theta)\left(J_f^i(\theta)\right)^\top$, we have
\be\begin{aligned} \|J_f^i(\theta) \nabla_f^2 \ell_i(\theta)(J_f^i(\theta))^\top \|_2 &\leq \|J_f^i(\theta) \|_2 \|\nabla_f^2 \ell_i(\theta)\|_2\|\left(J_f^i(\theta)\right)^\top \|_2   \\
 &\leq \|J_f^i(\theta) \|_F\| \nabla_f^2 \ell_i(\theta)\|_2\|\left(J_f^i(\theta)\right)^\top \|_F. \end{aligned}\ee
Since $J_f^i(\theta) = [\nabla f_1^i(\theta), \dots, \nabla f_m^i(\theta) ]$ and $\| \nabla f_j^i(\theta) \|_2\leq \kappa_f$, we have $\|J_f^i(\theta) \|_F \leq \sqrt{m} \kappa_f$ which implies $\|H_i(\theta) \|_2 \leq m \kappa_f^2  \tilde \kappa_\ell := \kappa_H.$ Similarly, we have
\be\begin{aligned} 
\|\Pi_i(\theta)\|_2 &= \|\sum_{j=1}^m \nabla_{f_j} \ell_i(\theta)\nabla_\theta^2f_j^i(\theta)\|_2\leq \sum_{j=1}^m \| \nabla_{f_j} \ell_i(\theta)\nabla_\theta^2f_j^i(\theta)\|_2\\
&\leq  \sum_{j=1}^m |\nabla_{f_j} \ell_i(\theta)| \| \nabla_\theta^2f_j^i(\theta)\|_2 \leq m \kappa_\ell \tilde \kappa_f := \kappa_\Pi.
\end{aligned}\ee

We next show that $J_f^i(\theta)$ is Lipschitz continuous. For each row of $J_f^i(\theta)$, we get:
 \[\nabla f_j^i(\theta_2) - \nabla f_j^i(\theta_1)  = \int_{0}^1 \nabla^2  f_j^i \left ((1-t)\theta_1+t(\theta_2)\right) \left(\theta_2 - \theta_1 \right ) dt.\]
 This implies $ \|\nabla f_j^i(\theta_2) - \nabla f_j^i(\theta_1) \|_2 \leq \frac{1}{2} \tilde \kappa_f \|\theta_2-\theta_1\|_2$ and 
\be\begin{aligned}
\|J_f^i(\theta_1) - J_f^i(\theta_2) \|_2 &\leq \|J_f^i(\theta_1) - J_f^i(\theta_2) \|_F \\
& \leq \frac{\sqrt{m }}{2}\tilde \kappa_f \|\theta_1-\theta_2\|_2\\
&:= L_J \|\theta_1-\theta_2\|_2 .\\
\end{aligned}\ee
Finally, we show that
$H_i(\theta)$ is Lipschitz continuous near $\theta^*$ with Lipschitz constant $L_H$:
	\be
	\begin{aligned}
	&\|H_i(\theta_1) - H_i(\theta_2)\|_2\\  =&   \| J_f^i(\theta_2)  \nabla_f^2 \ell_i(\theta_2)(J_f^i(\theta_2))^\top - J_f^i(\theta_1)  \nabla_f^2 \ell_i(\theta_2)(J_f^i(\theta_2))^\top \\ &+ J_f^i(\theta_1)  \nabla_f^2 \ell_i(\theta_2)(J_f^i(\theta_2))^\top-  J_f^i(\theta_1)  \nabla_f^2 \ell_i(\theta_1)(J_f^i(\theta_1))^\top\|_2 \\
	\leq & \|J_f^i(\theta_2)   - J_f^i(\theta_1)\|_2 \|  \nabla_f^2 \ell_i(\theta_2)\|_2\|(J_f^i(\theta_2))^\top \|_2 \\ &+  \|J_f^i(\theta_1)  \nabla_f^2 \ell_i(\theta_2)(J_f^i(\theta_2))^\top-  J_f^i(\theta_1)  \nabla_f^2 \ell_i(\theta_1)(J_f^i(\theta_1))^\top\|_2   \\
	\leq & L_f\|\theta_2 -\theta_1\|_2 \tilde \kappa_\ell \sqrt{m}\kappa_f +  \sqrt{m}\kappa_f \left( \| \nabla_f^2 \ell_i(\theta_2)(J_f^i(\theta_2))^\top -   \nabla_f^2 \ell_i(\theta_1)(J_f^i(\theta_1))^\top \|_2\right)\\
	\leq & L_f\|\theta_2 -\theta_1\|_2 \tilde \kappa_\ell \sqrt{m}\kappa_f +  \sqrt{m}\kappa_f \left( \tilde \kappa_\ell \| (J_f^i(\theta_2))^\top - (J_f^i(\theta_1))^\top\|_2 + \sqrt{m}\kappa_f\|\nabla_f^2 \ell_i(\theta_2) -   \nabla_f^2 \ell_i(\theta_1)\|_2\right)\\
	\leq & L_f\|\theta_2 -\theta_1\|_2 \tilde \kappa_\ell \sqrt{m}\kappa_f +  \sqrt{m}\kappa_f \left( \tilde \kappa_J L_\ell\| \theta_1-\theta_2\|_2 + \sqrt{m}\kappa_fL_\ell\|\theta_1-\theta_2\|_2\right)\\
	 :=& L_H \|\theta_1-\theta_2\|_2.
	\end{aligned}
	\ee
	\end{proof}
Our local results are summarized in the following theorem.
\begin{theorem2}
Suppose that Assumption \ref{local-assump-supp} is satisfied. If the sample size $|S_H^k|$ increases superlinearly,
then the sequence $\{\theta_k\}$ generated by \eqref{iteration-local-supp} converges to $\theta^*$ superlinearly almost surely.
\end{theorem2}
\begin{proof}
The proof is divided into two parts. The first part is to show that the stochastic Dennis-Mor\"e condition holds almost surely, i.e., 
\be \label{local-dennis-more} \lim_{k\rightarrow \infty} \frac{\|(B_k - \nabla^2 \Psi(\theta^*)s_k)\|_2}{\|s_k\|_2} = 0\  a.s..\ee
The second part is to show that we can obtain the superlinear convergence rate from \eqref{local-dennis-more}.

\textbf{(1.)} By Lemma \ref{local-lemma}, we have $\|H_{S_H^k}(\theta) - H(\theta)\|\leq 2\kappa_H$, $\|\Pi_{S_H^k}(\theta) - \Pi(\theta)\|\leq 2\kappa_\Pi$. The matrix Bernstein's inequality yields 
\[
\Prob(\| H_{S_H^k}(\theta) - H(\theta)\|_2\geq \epsilon_k ) \leq 2 n \exp\{-\frac{\epsilon_k^2 |S_H^k|}{16\kappa_1^2	}\} \ \text{and} \
\Prob(\| \Pi_{S_H^k}(\theta) - \Pi(\theta)\|_2\geq \epsilon_k ) \leq 2 n \exp\{-\frac{\epsilon_k^2 |S_H^k|}{16\kappa_2^2	}\} .
\]
By construction, let $\sum_{k=1}^\infty \epsilon_k < \infty$ and the sample size grow so that $\sum_{k=1}^\infty  2 n \exp\{-\frac{\epsilon_k^2 |S_H^k|}{16\kappa^2	}\} < \infty $. This can be guaranteed, for example, if we choose $\epsilon_k = O(\frac{1}{k^{1+\delta_1}})$ and $|S_H^k| = O(k^{3+3\delta_1})$.

By Borel-Cantelli Lemma, there exists $k_0$ such that $\forall k>k_0$, $\| H_{S_H^k}(\theta) - H(\theta)\|_2 \leq \epsilon_k$ a.s. and $\| \Pi_{S_H^k}(\theta) - \Pi(\theta)\|_2 \leq \epsilon_k$ a.s..
Define the space where $\sum_k \|\theta_k - \theta^*\| < \infty$, $\| H_{S_H^k}(\theta) - H(\theta)\|_2 \leq \epsilon_k$ and $\| \Pi_{S_H^k}(\theta) - \Pi(\theta)\|_2 \leq \epsilon_k$ by $\Xi$. It is easy to know that $\mathbb{P}(\Xi) = 1$. Denote $e_k = \max \{\|\theta_{k+1} - \theta^*\|, \|\theta_k - \theta^*\|\},$ and $\sum_{i=1}^\infty e_k < \infty$ in space $\Xi$.  

Define two hypothetical sequences:
\[
\begin{aligned}
\widehat{\Lambda}_{k+1} &= {\Lambda}_k  - \frac{\Lambda_ku_ku_k^\top\Lambda_k}{u_k^\top \Lambda_ku_k} + \frac{\Pi_{S_{H}^k}(\theta^*)u_ku_k^\top \Pi_{S_{H}^k}(\theta^*)}{u_k^\top \Pi_{S_{H}^k}(\theta^*)u_k},\\
\widetilde{\Lambda}_{k+1} &= {\Lambda}_k  - \frac{\Lambda_ku_ku_k^\top\Lambda_k}{u_k^\top \Lambda_ku_k} + \frac{\Pi(\theta^*)u_ku_k^\top \Pi(\theta^*)}{u_k^\top \Pi(\theta^*)u_k}.\\
\end{aligned}
\]
From Lemma C.14 \cite{zhou2017stochastic}, we have:
\[
\|\widetilde{\Lambda}_{k+1} - I \|_F^2 - \|\Lambda_{k} - I\|_F^2 = - \left [ \left(1-\frac{u_k^\top \Lambda_k\Lambda_k u_k}{u_k^\top \Lambda_ku_k}\right)^2 + 2 \left( \frac{u_k^\top \Lambda_k\Lambda_k\Lambda_ku_k}{u_k^\top \Lambda_ku_k} - \left(\frac{u_k^\top \Lambda_k\Lambda_ku_k}{u_k^\top \Lambda_ku_k}\right)^2  \right)\right].
\]
Without loss of generality, we assume that $\Pi(\theta^*) = I$, otherwise do linear transformation for variables by $\widetilde \theta = \Pi(\theta^*)^{1/2}\theta$.
We next need to show that $ \| \Lambda_{k} - I \| - \| \widetilde \Lambda_{k+1} - I \| \rightarrow 0.$

From section 4 in \cite{griewank1982local}, this is required to prove that $$\|\Lambda_{k+1} - \widetilde \Lambda_{k+1}\| \leq O(\epsilon_k + e_k).$$
From Lemma C.15 in \cite{zhou2017stochastic}, this is required to prove that there exists constants $c_1, c_2, c_3, c_4$ such that:
\begin{itemize}
\item[a.1.1)] $c_1 u_k^\top u_k \leq v_k^\top u_k \leq c_2 u_k^\top u_k$,
\item[a.1.2)] $\| \delta_k \| \leq  c_3 \|u_k\|e_k$,
\item[a.1.3)] $ \frac{v_k^\top \delta_k}{u_k^\top v_k} \leq c_4 e_k $,
\end{itemize}
where $\delta_k = \Pi_{S_H^k}(\theta^*) u_k - v_k$.


From Assumption \ref{local-assump-supp}, we can obtain that when $\theta_{k}$ nears $\theta^*$, there exists $c_1 < \frac{1}{2}\widetilde \lambda$ such that $v_k^\top u_k \geq c_1 u_k^\top u_k$. By Lemma \ref{local-lemma}, it is easy to know that 
$v_k^\top u_k \leq \|u_k\| \|v_k\| \leq  L_{J} \kappa_\ell \|u_k\|^2. $ Let $c_2 =  L_{J} \kappa_\ell$ and we prove a.1.1). Note that each $f_j^i$ is twice continuously differentiable, we have \be
\begin{aligned}
\delta_k &= \Pi_{S_H^k}(\theta^*) u_k - v_k\\
& = \frac{1}{|\mathcal{S}_H^{k}|} \sum_{i\in \mathcal{S}_H^{k}} \sum_{j=1}^m \int_{0}^1 \left(  \nabla_{f_j}
\ell_i(\theta^*)\nabla_\theta^2f_j^i(\theta^*) - \nabla_{f_j}\ell_i(\theta_{k+1}) \nabla^2  f_j^i \left ((1-t)\theta_k+t(\theta_{k+1})\right)\right) u_k dt.
\end{aligned}
\ee 
Since $\| \nabla_{f_j} \ell_i(\theta)\|_2 \leq  \kappa_{\ell}$, $\| \nabla_{f_j}^2 \ell_i(\theta)\|_2 \leq \tilde \kappa_{\ell}$ and $\| \nabla^2 f_j^i(\theta_1) - \nabla^2 f_j^i(\theta_2)) \|_2 \leq L_f \|\theta_1-\theta_2\|_2, \forall i,j$, we conclude that there exists constant $c_3$, such that a.1.2 holds.
a.1.3 follows from a.1.1 and a.1.2 immediately by Cauchy–Schwarz inequality.

By a.1.1), a.1.2) and a.1.3), following from Lemma C.15 in \cite{zhou2017stochastic}, we can prove that 
\be
\begin{aligned}
\|\Lambda_{k+1} - \widehat \Lambda_{k+1} \| & = \left \|-\frac{v_k\delta^\top +\delta v_k^\top +\delta\delta^\top}{u_k^\top v_k}  + \frac{v_k^\top\delta(v_kv_k^\top +v_k\delta^\top +\delta v_k^\top +\delta\delta^\top)}{u_k^\top v_k + v_k^\top \delta}\right \| \leq O(e_k),\\
\|\widehat \Lambda_{k+1} - \widetilde \Lambda_{k+1}\|& = \left \|-\frac{\widetilde v_k \hat \delta^\top + \hat \delta \widetilde v_k^\top + \hat \delta \hat\delta^\top}{u_k^\top \widetilde v_k}  + \frac{ \widetilde v_k^\top \hat \delta( \widetilde v_k \widetilde v_k^\top +\widetilde v_k\hat \delta^\top +\hat \delta \widetilde v_k^\top +\hat \delta \hat \delta^\top)}{u_k^\top {\widetilde v}_k + v_k^\top \hat \delta} \right \| \leq O(\epsilon_k),
\end{aligned}
\ee
where $\widehat v_k =  \Pi_{S_H^k}(\theta^*) u_k$, $\widetilde  v_k = \Pi(\theta^*) u_k$ and $\hat{\delta} = \widehat v_k -\widetilde v_k.$
This shows that
\[
\| \Lambda_{k+1} - \widetilde \Lambda_k  \|  \leq O(e_k + \epsilon_k).
\]
Following the same idea of section 4 in \cite{griewank1982local}, we have
\[
\lim_{k \rightarrow \infty} \frac{\|(\Lambda_k -I) u_k \|}{\|u_k\|} = 0\quad a.s. .
\]
Our previous results yield that: 
\be
\label{stochatic-dennis-more}
\begin{aligned}
&\lim_{k\rightarrow \infty} \frac{\|(B_k - \nabla^2 \Psi(\theta^*)s_k)\|}{\|s_k\|}\\
=&\lim_{k\rightarrow \infty}\frac{\|(H_{S_H^k} + \Lambda_k - \nabla \Psi(\theta^*))u_k\|}{\|u_k\|}\\ 
=&\lim_{k\rightarrow \infty} \frac{\|(H_{S_H^k}(\theta_k) - H(\theta_k) + H(\theta_k) - H(\theta^*)+ \Lambda_k - \Pi(\theta^*))u_k\|}{\|u_k\|} \\
\leq & \lim_{k\rightarrow \infty}  \frac{\|(H_{S_H^k}(\theta_k) - H(\theta_k)\| \|u_k\| +\| H(\theta_k) - H(\theta^*)\| \|u_k\|+ \|(\Lambda_k - \Pi(\theta^*))u_k\|}{\|u_k\|} = 0. 
\end{aligned}
\ee
The result \eqref{stochatic-dennis-more} is actually the stochastic Dennis-M\"ore condition. 

\textbf{(2.)} The next step is to show that superlinear convergence results are guaranteed if \eqref{stochatic-dennis-more} holds.
For simplicity of notations, we set
    \begin{eqnarray*}
        w_1^k &=& (B_k - \nabla^2 \Psi(\theta^*))(\theta^{k+1}-\theta^k), \\
        w_2^k &=& \nabla \Psi(\theta^{k+1}) - \nabla \Psi(\theta^k) - \nabla^2\Psi(\theta^*)(\theta^{k+1}-\theta^k).
    \end{eqnarray*}
    Then by \eqref{iteration-local-supp}, we have
    \bee
    \begin{aligned}
       & B_k(\theta^{k+1} - \theta^{k})  - \nabla^2\Psi(\theta^*)(\theta^{k+1}-\theta^k) =  -\nabla \Psi(\theta^k) - \nabla^2\Psi(\theta^*)(\theta^{k+1}-\theta^k).
        \end{aligned}
    \eee
    It follows that
    \[
        w_1^k - w_2^k  =
        -\nabla \Psi(\theta^{k+1}).
    \]
    Due to Assumptions 1-2, we have that $\|w_1^k\|/\|\theta^{k+1}-\theta^k\|$ and $\|w_2^k\|/\|\theta^{k+1}-\theta^k\|$ converges to $0$ almost surely.
    It follows that
    \be
    \label{mkconv0}
     m_k := \frac{\|-\nabla \Psi(\theta^{k+1})\|}{\|\theta^{k+1}-\theta^k\|} \rightarrow 0 \text{ almost surely}.
    \ee
    By the nonsingularity of $\nabla^2 \Psi(x^*)$ and the convergence of
    $\{\theta^k\}$, with probability $1$, there exists a constant $\xi$ such that
    \[
        \|\nabla \Psi(\theta^{k+1})\| \geq \xi\|\theta^{k+1}-\theta^*\|.
    \]
    It implies that
    \bee
    \begin{aligned}
        m_k       &  \geq \frac{\xi\|\theta^{k+1}-\theta^*\|}{\|\theta^{k+1}-\theta^*\| + \|\theta^{k}-\theta^*\|}.
      \end{aligned}
    \eee
 Hence, it follows that  $$\frac{\|\theta^{k+1}-\theta^*\|}{\|\theta^{k}-\theta^*\|}\leq\frac{m_k}{\xi - m_k} \rightarrow 0.$$
 This finishes the proof.
\end{proof}

%

\end{document}